\documentclass[12pt,a4paper]{amsart}
\usepackage{amssymb}
\usepackage{amsfonts}
\usepackage{amsthm}
\usepackage{amsmath}
\usepackage{amscd}
\usepackage[utf8]{inputenc}
\usepackage{t1enc}
\usepackage[mathscr]{eucal}
\usepackage{indentfirst}
\usepackage{graphicx}
\usepackage{graphics}
\numberwithin{equation}{section}
\usepackage[margin=2.9cm]{geometry}
\usepackage{epstopdf} 
\usepackage{xcolor}

\usepackage{verbatim}
\usepackage{enumitem}
\usepackage[numbers,sort&compress]{natbib}
\usepackage{hyperref}

\theoremstyle{plain}
\newtheorem{theorem}{Theorem}[section]
\newtheorem{lemma}[theorem]{Lemma}
\newtheorem{corollary}[theorem]{Corollary}
\newtheorem{proposition}[theorem]{Proposition}

\theoremstyle{definition}
\newtheorem{definition}[theorem]{Definition}

\newtheorem{remark}[theorem]{Remark}
\newtheorem{?}[theorem]{Problem}
\newtheorem{example}[theorem]{Example}
\newcommand{\supp}{\operatorname{supp}}

\newcommand{\R}{\mathbb{R}}

\newcommand{\Z}{\mathbb{Z}}

\newcommand{\Nb}{\mathbf{N}}
\newcommand{\ab}{\mathbf{a}}
\newcommand{\mub}{\pmb{\mu}}
\newcommand{\hb}{\mathbf{h}}
\newcommand{\Hb}{\mathbf{H}}
\newcommand{\Bb}{\mathbf{B}}
\newcommand{\K}{\mathbb{K}}

\newcommand{\Nc}{\mathcal{N}}

\let\emptyset\varnothing

\title[Negative eigenvalue estimates]{Negative eigenvalue estimates for the 1D Schr{\"o}dinger operator with measure-potential}
\author[]{Robert Fulsche}
\address {Institut f\"ur Analysis,
	Leibniz Universit\"at Hannover, 
	Welfengarten 1, 30167 Hannover, Germany}
\email{fulsche@math.uni-hannover.de}

\author[]{Medet Nursultanov}
\address {Department of Mathematics and Statistics, University of Helsinki, Helsinki, and Institute of Mathematics and Mathematical Modeling, Almaty, Kazakhstan}
\email{medet.nursultanov@gmail.com}

\author[]{Grigori Rozenblum}
\address{Chalmers University of Technology}
\email{grigori@chalmers.se}

\begin{document}
	
	\begin{abstract}
		We investigate the negative part of the spectrum of the operator $-\partial^2 - \mu$ on $L^2(\mathbb R)$, where  a locally finite Radon measure $\mu \geq 0$  serves as a potential. We obtain estimates for the Eigenvalue counting function, for individual eigenvalues and estimates of the Lieb-Thirring type. A crucial tool for our estimates is Otelbaev's function, a certain average of the measure-potential $\mu$, which is used both in the proofs and the formulation of most of the results.
	\end{abstract}
	
	\maketitle

	\tableofcontents

	\section{Introduction} 
	The paper is devoted to obtaining negative spectrum estimates, in particular, Lieb-Thirring type estimates, for  Sturm-Liouville operators with  Radon measure serving as potential.
	
	\subsection{Eigenvalue estimates}Eigenvalue estimates for Schr{\"o}dinger type operators is a classical topic in analysis. For Schr{\"o}dinger operators of the form 
	\begin{equation}\label{intr.1}
		\mathbf{H}=-\Delta-V(x), \quad x\in\mathbb{R}^{\mathbf{N}},
	\end{equation}
	with the `potential' $V(x)$ tending to zero at infinity, there are numerous results concerning the number of negative eigenvalues, the sum of their powers and their distribution (when there are infinitely many of them). A huge, but still not exhausting, bibliography can be found in the recent book \cite{FrankLaptevWeidl}. Of special interest is the Lieb-Thirring inequality, relating the spectral properties of the Schr{\"o}dinger operator and properties of the corresponding classical Hamiltonian. It  has, in its most classical  case, the form
	\begin{equation}\label{LT.initial}
		\mathrm{LT}_{\gamma}(\mathbf{H})\equiv \sum|\lambda_\nu(\mathbf{H})|^\gamma\le C\int_{\mathbb{R}^\mathbf{N}}V_+(x)^{\frac{\mathbf{N}}{2}+\gamma}dx,
	\end{equation}
	where $\lambda_\nu(\mathbf{H})$ are the negative eigenvalues and $C=C_{\gamma,\Nb}$ is a constant not depending on $V$. The inequality \eqref{LT.initial} holds for $\gamma\ge 0$ if $\mathbf{N}\ge3$ (for $\gamma =0$, it is called the CLR inequality), for $\gamma>0$ if $\mathbf{N}=2$, and for $\gamma\ge\frac12$ if $\mathbf{N}=1$.

	Among many possible generalizations and versions of eigenvalue inequalities, one line of research consists in finding versions of these estimates for the case when the potential $V(x)$, classically supposed to be a function on $\mathbb{R}^\mathbf{N}$, is, in fact, more singular, e.g., a Radon measure on $\mathbb{R}^\mathbf{N}$. As a special case, one can consider measures which are singular with respect to the Lebesgue measure.
	
	In physical and mathematical literature, such singular potentials have attracted considerable interest. A huge bibliography can be found in \cite{SavchukShkalikov, Albev.Gesztesy, MikhailetsMolyboga, DemkovOstrovskii, BrascheExnerKuperin, BelloniRobinett, AlKosh, BrascheNizhnik, KostenkoMalamud, Bonin} and many more. Of course, the one-dimensional case was studied in most detail.
	
	As it concerns the quantitative spectral analysis for general measures serving as a potential, one of the complications is the absence of a quasi-classical pattern which usually hints for the expression entering into the eigenvalue inequalities and asymptotics - compared, say,  with \eqref{LT.initial}, where both the dependence on the potential $V$ and even the best possible coefficient in the estimate are governed by certain quasi-classical heuristics. So, although the qualitative spectral analysis of such Schr{\"o}dinger operators was satisfactory, the question on spectral estimates was rather unexplored.

	\subsection{Singular potentials}
	Probably the first rigorous quantitative spectral study involving singular potentials was \cite{FrankLaptev2008}, where, for a potential supported on a hyperplane in $\mathbb{R}^\mathbf{N}$, $V(x',x_\mathbf{N})=v(x')\otimes\delta(x_{\mathbf{N}})$, the estimate, similar to \eqref{LT.initial}, for the quantity $\mathrm{LT}_{\gamma}(\mathbf{H})$ was found, with the integral of $v(x')^{\mathbf{N}+2\gamma}$ on the right-hand side.
	
	Quite recently, in \cite{Rozenblum2022,RozenblumTashchiyan}, an approach was developed to establish eigenvalue estimates and LT type inequalities for singular potentials of the form $\mu=v(x)\varpi$, where $\varpi$ is a fixed  measure satisfying the `upper Ahlfors condition', 
	\begin{equation}\label{upper Ahlfors}
		\varpi(B(x,r))\le A r^{s},\quad  s>0,
	\end{equation}
	or its two-sided version, and $v(x)\ge0,$ $x\in\supp\varpi$ is a weight function.
	Here, $B(x,r)$ is the ball with radius $r$ centered at $x$. Note that $s=\Nb$ corresponds to an absolutely continuous measure and thus to the usual Schr{\"o}dinger operator \eqref{intr.1}. The estimate in \cite{Rozenblum2022}  has the form 
	\begin{equation}\label{LT.GR}
		\mathrm{LT}_\gamma(\mathbf{H})\leq C \int v(x)^{\theta}\varpi(dx),
		\qquad
		\theta= \frac{s+2\gamma}{s+2-\mathbf{N}},\quad  s>\mathbf{N}-2, s>0,
	\end{equation}
	where $\gamma>0$ for $\mathbf{N}\ge 2$, $\gamma\ge \frac12$ for $\mathbf{N}=1$, in particular, reproducing the above result by R.\ Frank and A.\ Laptev, where $s=\mathbf{N}-1$ (the case $\gamma=0$ and $\mathbf{N}\ge2$ was taken care of previously in \cite{RozenblumTashchiyan}). 
	
	In particular, in the one-dimensional case, $\mathbf{N}=1,$ which we are going to discuss further in the present paper, the conditions in \eqref{LT.GR} allow $\gamma=\frac12$ as long as $s>0$. For this special value of the parameter $\gamma,$ thus, $\theta=1$, $\int v\varpi(dx)=\mu(\mathbb{R}^1),$ estimate \eqref{LT.GR} takes the form
	\begin{equation}\label{LT.GR.1/2}
		\mathrm{LT}_{\frac12}(\mathbf{H})\le C_1\mu(\mathbb{R}^1).
	\end{equation}
	This aesthetic estimate, in which the measure $\mu$ itself enters (and not its density),  creates an impression that, hopefully, \eqref{LT.GR.1/2} is valid for $s=0,$ this means, for an arbitrary nonnegative measure $\mu$, not necessarily an absolutely continuous, as well. The proof in \cite{Rozenblum2022} breaks down for $s=0$, however, fortunately, this case was taken care of previously, in \cite{HLT}, moreover, with sharp constant $C_1=\frac12$ in \eqref{LT.GR.1/2}.  For other values of $\theta,$ $\theta\ne \frac12$,  estimate \eqref{LT.GR.1/2} does not admit a natural generalization, moreover, it is even hard to guess proper terms. 

Generally,	the spectra of  Schr\"{o}dinger (Sturm-Liouville) operators with singular potential were being studied for a long time, especially, in  cases where the potential is a purely discrete measure. Such operators appear in various mathematical models, interesting by themselves but also as an approximation to `continuous' models, and they are still an important field of research today. Here, many methods specific to the one-dimensional case can be applied. 	However, in these studies, the eigenvalue properties are usually expressed rather indirectly, via zeros and poles of some analytic functions generated by the problem or through continued fractions, rather than directly in terms of the characteristics of the potential.
	
	Therefore, we considered it a challenging problem to find estimates for the distribution of negative eigenvalues and for  $\mathrm{LT}_\gamma(\mathbf{H})$ with general, possibly singular and containing a discrete component, measures acting as a potential.  The question on a discrete component, to the best of our knowledge, present only in dimension 1. In a higher dimension, a complication arises when defining a self-adjoint operator; say  with $\delta$-potentials in $\mathbb{R}^{\mathbf{N}}, $ $\mathbf{N}\ge2$, one needs to choose among possible self-adjoint realisations, and this choice, of course, may influence the spectrum considerably. Formally, in \eqref{LT.GR}, discrete measures are excluded by the condition $s>\mathbf{N}-2$.
	
	The feature, common for all dimensions, is the following. Usually, eigenvalue estimates, including \textbf{LT}-estimates like \eqref{LT.initial} or \eqref{LT.GR}, bound the eigenvalues via integral characteristics of the potential $V$; this means, in particular, that the estimates are the same for equimeasurable potentials and, therefore, do not capture the spacial distribution of the potential. Such dependence needs to be reflected by some different form of the right-hand side of the estimates. The natural idea is to use a kind of averaging of the potential in order to obtain sharper eigenvalue estimates. One of the possible choices of averaging, which we use in this paper, is the instrument created long ago by M.\ Otelbaev for finding eigenvalue estimates for the one-dimensional Schr{\"o}dinger operator in the case when the potential, although a function, but quite singular, tends to $+\infty$ at infinity, so that the whole spectrum is discrete, see \cite{Otelbaev,Otelbaev2}. The idea of M.\ Otelbaev's method involves constructing a certain function, the "Otelbaev function", which a certain variable window  average   of the singular potential. In this way, Otelbaev's function reflects both the size of the measure-potential and the spreading of this measure along the real line.
	
	Recently, the two first-named authors of this paper applied this approach by M. Otelbaev to finding sharp eigenvalue estimates for the Sturm-Liouville operator with a measure serving as a potential, again, in the case when the whole spectrum is discrete, see \cite{FulscheNursultanov,Nursultanov}.
	
	In the present paper, we study the spectrum in the case  of the potential tending to zero at infinity (in some sense). This case  is somewhat more delicate than the one above due to the presence of the essential spectrum filling the positive semi-axis. 
	
	We start with some general properties of Radon measures on $\R^1$ and discuss criteria for discreteness of the negative spectrum of the Schr{\"o}dinger operator with such measure as the potential. After this,  we introduce a proper version of the Otelbaev function and describe its basic properties, such as its dependence on the averaging parameter $\alpha$ and behavior at infinity. These properties enable us further on to find \emph{two-sided } eigenvalue estimates for the distribution of negative eigenvalues for the potential being a Radon measure. The measure is not necessarily discrete; however, it may be. In particular, LT-type estimates are obtained in terms of the Otelbaev function, providing both upper and lower bounds, including direct and reverse \textbf{LT}-type inequalities. The latter inequalities cover all positive values of $\gamma$ and   are sharp in the sense that both upper  and lower estimates involve one and the same expression containing the Otelbaev function.  For upper estimates we follow initial approach by T.Weidl, \cite{Weidl}, where  bracketing is used, with the size of partial intervals adapted to the measure $\mu$. Note that reverse \textbf{LT}  estimates, see  \cite{FrankLaptevWeidl}, Theorem 4.48, hold only for $\gamma\le \frac12,$ while ours  are valid for all values of $\gamma>0$.  Note, however, that for $\gamma=\frac12$, the result in \cite{Schmincke}  gives a sharp value of constant in the reverse    \textbf{LT} estimate. However, the approach in \cite{Schmincke} to the lower \textbf{LT} estimate does not easily extend to measure potentials.
	
	Further on, in a number of examples, detailed calculations show effects of various properties of the measure influencing the eigenvalue distribution. We show that our estimates turn out to be sharper than the ones in \cite{Weidl}, and we also show that our estimates improve the subcritical ($\gamma<\frac12$) ones in \cite{NetrusovWeidl}. Additionally, in the case when the condition of discreteness of the negative spectrum is not assumed, we find, in terms of the Otelbaev function, two-sided estimates for the lowest point of the spectrum and the lowest point of the essential spectrum of the operator $\Hb.$ 

	In the paper \cite{BFS23} and the preprint \cite{BFS24},  which appeared recently, a different approach to spectral estimates, including \textbf{LT} ones,  for nonregular potentials is developed, based upon the "landscape" function. Some results, although less explicit, look similar to ours, however the conditions on the operator are considerably more restrictive, requiring, in particular, the potential to be a function in the Kato class.

	\section{Preliminaries}
	\subsection{Radon measures on $\mathbb{R}^1$}
	Throughout this paper, $\mathcal{L}(A)$ denotes the Lebesgue measure of the Borel set $A \subset \mathbb R$. For any interval $I$ in $\mathbb{R}$, the notation $|I|$ refers to the length of the interval. For a set $A\subset \mathbb{R}$, we denote its closure as $\overline{A}$. 
	We let $\mathcal{D}(\cdot)$ and $\mathfrak{d}(\cdot)$ denote the domain of an operator and a form, respectively. By $\sigma(\cdot)$ and $\sigma_{ess}(\cdot)$, we denote the spectrum and the essential spectrum of an operator, respectively. By $\delta_x$, we denote the Dirac measure concentrated at $x \in \mathbb R$. For $x \in \mathbb R$, by $[x]$, we mean the largest integer not exceeding $x$. For $x\in \mathbb{R}$, we use the notation 
	\begin{equation*}
		\Delta_x(d) = \left[x - \frac{d}{2},x + \frac{d}{2}\right].
	\end{equation*}

	Recall that a (positive) locally finite measure $\mu$ on the Borel-$\sigma$-algebra $\mathcal B(\mathbb R)$ of $\mathbb R$ is a \emph{Radon measure} if it satisfies the conditions of local finiteness, outer regularity and inner regularity:
	\begin{align*}
		\mu(K) &< \infty \text{ for all } K \subset \mathbb R \text{ compact,}\\
		\mu(K) &= \inf \{ \mu(O); O \supset K, ~ O \text{ open}\} \text{ for all } K \subset \mathbb R \text{ compact,}\\
		\mu(O) &= \sup \{ \mu(K); K \subset O, ~ K \text{ compact}\} \text{ for all } O \subset \mathbb R \text{ open.}
	\end{align*}

	Later, we will use the following simple lemma:
	\begin{lemma}\label{loc_finite}
		Let $\mu$ be a non-negative Radon measure on $\mathbb R$ and $x\in \mathbb{R}$. Then, $\mu([x - \tau,x)) \rightarrow 0$ as $\tau \rightarrow 0$.
	\end{lemma}
	\begin{proof}
		For $\varepsilon>0$, the inner regularity says that
		\begin{equation*}
			\mu((x-2\tau,x)) = \sup \{ \mu(K); K \subset (x-2\tau,x), ~ K \text{ compact}\}.
		\end{equation*}
		Therefore, there exists a compact set $K_\varepsilon\subset (x-2\tau,x)$ such that $\mu((x-2\tau,x)) - \mu(K_\varepsilon) < \varepsilon$. Since $K_\varepsilon \subset (x-2\tau,x)$ is compact, it follows for sufficiently small $\tau_\varepsilon$ that $[x-\tau_\varepsilon,x) \cap K_\varepsilon = \emptyset$. Therefore, $[x-\tau_\varepsilon,x)\subset (x-2\tau,x) \setminus K_\varepsilon$, and consequently, $\mu([x-\tau_\varepsilon,x))<\varepsilon$. 
	\end{proof}
	
	\subsection{Formulation of the problem}\label{formulation}
	Let $\mu$ be a non-negative Radon measure on $\mathbb R$. We consider the following sesquilinear form
	\begin{equation}\label{form}
		\mathbf{a}_\mu(f,g)=\int_{\mathbb{R}}f'\overline{g'} ~dx - \int_{\mathbb{R}} f\overline{g}~d\mu
	\end{equation}
	with the domain
	\begin{equation*}
		\mathfrak{d}(\mathbf{a}_\mu)=\left\{f\in H^1(\mathbb{R}):  \int_{\mathbb{R}} |f|^2 d\mu < \infty \right\};
	\end{equation*}
	the corresponding quadratic form will be denoted by $\mathbf{a}_\mu[f]:=\mathbf{a}_\mu(f,f)$.
	We recall  the Brinck condition for  $\mu$, see \cite{Brinck}: There exists $C>0$ such that 
	\begin{equation}\label{Brinck}
		\mu([x,x+1]) < C \quad \text{for any } x\in \mathbb{R}.
	\end{equation}
	It is known that \eqref{Brinck} guarantees that the form $\mathbf{a}_\mu$ is lower semi-bounded and closed; see \cite{FulscheNursultanov} or \cite{MikhailetsMolyboga}. Since $\mu$ is sign-definite, the converse is also true: 
	\begin{lemma}\label{semiboundedness}
		Let $\mu$ be a non-negative Radon measure. The following statements are equivalent:
		\begin{enumerate}
			\item \label{semi_boud_mesure} The Brinck condition \eqref{Brinck} holds.
			\item \label{semi_boud_form} The sesquilinear form $\mathbf{a}_\mu$ is lower semi-bounded and closed. 
		\end{enumerate}
	\end{lemma}
	\begin{proof}
		Due to \cite[Prop.\ II.4]{FulscheNursultanov} or \cite{MikhailetsMolyboga}, it remains to show that \eqref{semi_boud_form} implies \eqref{semi_boud_mesure}. Assuming that \eqref{semi_boud_mesure} does not hold, for any fixed  $C>0$, there exists $x\in \mathbb{R}$, say, $x=\frac12$, such that $\mu([x,x+1])\geq 3C + \pi^2$.  We define
		\begin{eqnarray*}
			u(x) := \begin{cases}
				1-\cos\left( \pi\left(x-2\right)\right), ~&x \in [0,2],\\
				0, & x \not \in [0,2].
			\end{cases}
		\end{eqnarray*}
		Then, $\|u\|^2 = 3$, while $u \geq 1$ on $[\frac 12, \frac 32]$. Therefore, 
		\begin{equation*}
			\mathbf{a}_\mu[u] = \pi^2 - \int_{0}^2 u^2d\mu \leq \pi^2 - \mu\left (\left [\frac 12 , \frac 32 \right ] \right ).
		\end{equation*}
		Hence,
		\begin{equation*}
			\mathbf{a}_\mu[u]\leq \left(\frac{\pi^2}{3} - \frac{\mu([1/2,3/2])}{3} \right)\|u\|^2 \leq -C\|u\|^2.
		\end{equation*}
		Since $C>0$ was arbitrary, we conclude that $\mathbf{a}_\mu$ is not semi-bounded. 
	\end{proof}
	From now on, we will always assume that $\mu$ satisfies  \eqref{Brinck}, so that $\mathbf{a}_\mu$ is semi-bounded, closed, symmetric, and densely defined. We investigate the
	spectral properties of the operator $\Hb_\mu$ associated with $\mathbf{a}_\mu$ in the sense of Kato's first representation theorem \cite[Theorem 2.6]{Kato}.
	\subsection{Discreteness of the  negative spectrum of $\Hb_{\mu}$} Spectral properties of the operator $\Hb_\mu$  for an absolutely continuous measure $\mu$ are well known, see, e.g., \cite{Birman.61}. For our more general case, we refer to \cite{MazVerb1D}. There, the measure $\mu$ is even allowed to be non-sign-definite.
	
	We recall the general abstract criterion for discreteness of the negative spectrum of Schr{\"o}dinger-like operators; see Theorem 1.3 in  \cite{Birman.61}. Being applied to the operator $\Hb_{\epsilon,\mu}$ defined by the lower semibounded quadratic form
	
	\begin{equation*}
		\hb_{\epsilon,\mu}[u]=\ab[u]-\epsilon\mub[u], \ \ab[u]=\int_{\R^1}|u'(x)|^2dx, \ \mub[u]=\int_{\R^1} |u(x)|^2 d\mu(x),
	\end{equation*}
	on its natural domain,  this theorem establishes that  the negative spectrum of $\Hb_{\epsilon,\mu}$ is discrete for all $\epsilon>0$ if and only if the quadratic form 
	$\mub[u]$ is compact in the Sobolev space $H^1(\R^1)$. 
	
	The necessary and sufficient condition of this compactness is contained in Theorem 2.5 in \cite{MazVerb1D}:
	\begin{theorem}\label{Maz}
		The quadratic form $\mub$ is compact in the Sobolev space $H^1(\R^1)$ if and only if, for any fixed $a>0$, the functions
		\begin{equation*}
			G_+(t)=e^t\int_t^{\infty}e^{-s}d\mu(s), \, G_{-}(t)=e^{-t}\int_{-\infty}^t e^{s}d\mu(s)
		\end{equation*}
		satisfy
		\begin{equation*}
			\lim_{x\to\pm\infty}\int_{x-a}^{x+a}G_{\pm}(t)^2 dt=0.
		\end{equation*}
	\end{theorem}
	\begin{remark}\label{rem_discr}
		For a positive measure, this condition is equivalent to 
		\begin{equation}\label{MazCond}
			\mu((x-a,x+a))\to 0 \quad \mbox{as} \, x\to\pm\infty
		\end{equation}
		for some, and therefore, for every $a>0.$ 
	\end{remark}
	
	We will always suppose that these conditions are fulfilled.
	\begin{remark}
		Theorem \ref{Maz} is valid for a non-sign-definite measure $\mu$ as well. Finding eigenvalue estimates for such measures, even for absolutely continuous ones,   while one needs to trace the possible cancellation of the contributions of the positive and negative components, is a hard and challenging problem. See \cite{DaHuSi} for some partial results.    
	\end{remark}
	
	\section{Otelbaev's function}
	In this section, we introduce Otelbaev's function and establish its properties. This function, introduced in \cite{Otelbaev} for the study of Schr{\"o}dinger operators on $\R^1$ having purely discrete spectrum, is a special kind of averaging of the potential. This construction was extended to measure-potentials in \cite{FulscheNursultanov,Nursultanov}.
	\begin{definition}\label{def}
		Let $\mu$ be a non-negative Radon measure on $\mathbb R$; we fix the averaging parameter $\alpha>0$. Then, we define the Otelbaev function by
		\begin{equation}\label{Otelb.function}
			q^*_\alpha(x) := q_{\alpha, \mu}^\ast(x) :=  \inf \left\{\frac1{d^2}: \mu\left(\Delta_x(d)\right) < \frac{1}{\alpha d} \right\}.
		\end{equation}
	\end{definition}
	We will write $q_{\alpha, \mu}^\ast(x)$ or  $q_\alpha^\ast(x)$  if we want to emphasize the dependence on the particular measure $\mu$ or, otherwise, if no confusion between different measures is possible.
	
	Obviously, for any fixed $x$ and sufficiently small $d>0$, the inequality in \eqref{Otelb.function} holds; for $d$ large enough, this inequality does not hold, provided $\mu$ is not the zero measure. By monotonicity in $d$, $q_\alpha^*$ is well-defined. Moreover, we will show that $q_\alpha^*$ is continuous, and if $\mu(\mathbb{R})>0$, then it is strictly positive. We will also consider the auxiliary quantity
	\begin{equation*}
		d_\alpha(x) := \frac{1}{\sqrt{q_\alpha^*(x)}}.
	\end{equation*}
	\begin{remark}\label{rem_cont_meas}
		If $\mu$ is a\emph{ continuous} measure, that is, it contains no point masses, then $d_\alpha(x)$ is the unique solution to the equation
		\begin{equation}\label{eq}
			\mu\left(\Delta_x(d)\right) = \frac{1}{\alpha d}.
		\end{equation}
	\end{remark}
	
	If $\mu$ contains point masses, this statement may be wrong. Indeed, take $\mu = \delta_{-1} + \delta_{1}$.
	Then, for any $0<\tau< 2$ and $\kappa\geq 2$,
	\begin{equation*}
		\mu\left(\left[-\frac{\tau}{2},\frac{\tau}{2}\right]\right) = 0 < \frac{1}{2\tau},
		\qquad
		\mu\left(\left[-\frac{\kappa}{2},\frac{\kappa}{2}\right]\right) = 2 > \frac{1}{2\kappa}.
	\end{equation*}
	Therefore, $q_1^*(0) = 1/4$ and $d_1(0) = 2$. Since
	\begin{equation*}
		\mu\left(\left[-\frac{d_1(0)}{2},\frac{d_1(0)}{2}\right]\right) = \mu([-1,1]) > 1/2,
	\end{equation*}
	we conclude that $d_1(0)$ does not satisfy \eqref{eq} with $x=0$. 
	
	In general, $d_\alpha(x)$ is the length of the smallest interval centered at $x$ such that its $\mu$-measure is not less than $1/(\alpha d_\alpha(x))$. More formally, this property can be stated as follows:
	\begin{lemma}\label{max_interval}
		Let $\alpha>0$ and $x\in \mathbb{R}$. Then 
		\begin{equation*}
			\mu\left( \Delta_x(d_\alpha(x)) \right) \geq \frac{1}{\alpha d_\alpha(x)}
		\end{equation*}
		and 
		\begin{equation*}
			\mu\left( \Delta_x(d_\alpha(x) - \varepsilon) \right) < \frac{1}{\alpha (d_\alpha(x) - \varepsilon)}
		\end{equation*}
		for any $0 < \varepsilon < d_\alpha(x)$.
	\end{lemma}

	\begin{proof}
		Assume that the first estimate breaks down, that is,  $ \mu\left( \Delta_x(d_\alpha(x)) \right) < \frac{1}{\alpha d_\alpha(x)}.$  By Lemma \ref{loc_finite}, 
		\begin{equation*}
			\mu\left(\left[x - \frac{d_\alpha(x) + \varepsilon}{2}\right., \left. x - \frac{d_\alpha(x)}{2}\right) \right) \rightarrow 0,\, 
			\mu\left( \left(x + \frac{d_\alpha(x) }{2}\right., \left. x + \frac{d_\alpha(x)+ \varepsilon}{2}\right]\right) \rightarrow 0,
		\end{equation*}
		as $\varepsilon \rightarrow 0$. Therefore, $\mu\left( \Delta_x(d_\alpha(x)+\varepsilon) \right) \rightarrow \mu\left( \Delta_x(d_\alpha(x)\right)$. Since
		$  \frac{1}{\alpha (d_\alpha(x)+\varepsilon)} \rightarrow \frac{1}{\alpha d_\alpha(x)},    $
		we obtain, due to our assumption, for sufficiently small $\varepsilon>0$:
		\begin{equation*}
			\mu\left( \Delta_x(d_\alpha(x)+\varepsilon) \right) < \frac{1}{\alpha (d_\alpha(x)+\varepsilon)}.
		\end{equation*}
		Hence, by the  definition of $q_\alpha^*$, we obtain 
		$    \frac{1}{(d_\alpha(x) + \varepsilon)^2} \geq q_\alpha^*(x)$, and this contradicts $q_\alpha^*(x) = 1/ d_\alpha^2(x)$.
		
		If  the second statement of the lemma is wrong,  there should exist $\varepsilon\in (0,d_\alpha(x))$ such that
		\begin{equation*}
			\mu\left( \Delta_x(d_\alpha(x)-\varepsilon) \right) \geq \frac{1}{\alpha (d_\alpha(x) - \varepsilon)}.
		\end{equation*}
		Then, for any $d\geq d_\alpha(x) - \varepsilon$, 
		\begin{equation*}
			\mu\left( \Delta_x(d) \right)\geq \mu \left( \Delta_x(d_\alpha(x)-\varepsilon) \right) \geq \frac{1}{\alpha (d_\alpha(x) -\varepsilon)} \geq \frac{1}{\alpha d}.
		\end{equation*}
		Therefore, 
		\begin{align*}
			q^*_\alpha(x) &= \inf \left\{\frac1{d^2}: \mu\left(\Delta_x(d)\right) < \frac{1}{\alpha d} \right\}\\
			&= \inf \left\{\frac1{d^2}: \mu\left(\Delta_x(d)\right) < \frac{1}{\alpha d} \; \text{and} \; d< d_\alpha(x) - \varepsilon\right\}.
		\end{align*}
		In particular, 
		\begin{equation*}
			q_\alpha^*(x) \geq \frac{1}{(d_\alpha(x) - \varepsilon)^2},
		\end{equation*}
		and this contradicts $q_\alpha^*(x) = 1/ d_\alpha^2(x)$. 
	\end{proof}
	
	In \cite{FulscheNursultanov}, Otelbaev's function was defined in a slightly different way, i.e., in the \emph{infimum} defining $q_\alpha^\ast$, the condition $\mu(\Delta_x(d)) \leq \frac{1}{\alpha d}$ was used instead of the present condition with the strict inequality. We will briefly verify that the definition in \cite{FulscheNursultanov} and the present one are equivalent; the present one is more convenient.
	\begin{lemma}
		The following equalities hold true:
		\begin{equation*}
			q^*_\alpha(x) = \inf \left\{\frac1{d^2}: \mu\left(\Delta_x(d)\right) \leq \frac{1}{\alpha d} \right\}
		\end{equation*}
		and
		\begin{equation}\label{old_def}
			d_\alpha(x) = \sup \left\{d: \mu\left(\Delta_x(d)\right) \leq \frac{1}{\alpha d} \right\}.
		\end{equation}
	\end{lemma}
	\begin{proof}
		Due to the relation between $q^*_\alpha$ and $d_\alpha$, we only need to show the second statement. Further, it suffices to prove the statement for $\alpha = 1$; so we denote the right-hand side of \eqref{old_def} by $\tilde{d}_1(x)$. For $x\in \mathbb{R}$, obviously, $d_1(x)\leq \tilde{d}_1(x)$. To show the converse estimate, we choose any $d>d_1(x)$. Then, by Lemma \ref{max_interval},
		\begin{equation*}
			\mu\left(\Delta_x(d)\right) \geq \mu\left(\Delta_x(d_1(x))\right)\geq \frac{1}{d_1(x)} > \frac{1}{d}
		\end{equation*}
		for any $d>d_1(x)$. Therefore, 
		\begin{align*}
			\tilde{d}_1(x) &= \sup \left\{ d > 0: ~\mu\left(\Delta_x(d)\right) < \frac{1}{d}\right\}\\
			&= \sup \left\{ d > 0: ~\mu\left(\Delta_x(d)\right) \leq \frac{1}{d}, \; d\leq d_1(x) \right\}.
		\end{align*}
		This implies $\tilde{d}_1(x)\leq d_1(x)$. Hence, we  conclude that $\tilde{d}_1(x)=d_1(x)$.
	\end{proof}
	
	Now, we are ready to derive some crucial properties of the function $q_\alpha^*$.
	\begin{proposition}
		Let $\mu$ be a non-negative Radon measure on $\mathbb R$. The following statements hold true:
		\begin{enumerate}
			\item \label{q*_positive}If $\mu \neq 0$, then $q_\alpha^\ast(x) > 0$ for every $x \in \mathbb R$.
			\item \label{q*_joint_cont} $q_\alpha^\ast(x)$ is a jointly continuous function of $(\alpha, x) \in (0, \infty) \times \mathbb R$.
		\end{enumerate}
	\end{proposition}
	\begin{proof}
		By the definition, $q_{\alpha, \mu}^\ast(x) = q_{1, \alpha \mu}^\ast(x)$. Hence, \eqref{q*_positive} follows from the statement for $\alpha = 1$, cf.\ \cite[Prop.\ III.3]{FulscheNursultanov}. To verify \eqref{q*_joint_cont}, note that whenever there is a sequence $(\alpha_n, t_n)_{n \in \mathbb N}$ such that  $(\alpha_n, t_n) \to (\alpha, t)$, then $\alpha_n \mu(\cdot - t_n) \to \alpha \mu(\cdot - t)$ in weak$^\ast$ topology, i.e., for any $f \in C_c(\mathbb R),$   by the dominated convergence theorem,
		\begin{align*}
			\int &f(x) ~d(\alpha_n \mu(\cdot - t_n))(x) = \alpha_n \int f(x + t_n)~ d\mu(x) \\
			&\overset{n \to \infty}{\longrightarrow} \alpha \int f(x + t) ~d\mu(x) = \int f(x) ~d(\alpha \mu(\cdot - t))(x).
		\end{align*}
		Hence, by \cite[Prop.\ III.5]{FulscheNursultanov},  for each $x \in \mathbb R$,
		\begin{align*} 
			q_{\alpha_n, \mu}^\ast(x - t_n) = q_{1, \alpha_n \mu(\cdot - t_n)}^\ast(x) \longrightarrow q_{1, \alpha\mu(\cdot - t)}^\ast(x) = q_{\alpha, \mu}^\ast(x-t), 
		\end{align*}
		proving that $q_{\alpha, \mu}^\ast(x)$ is jointly continuous in $\alpha$ and $x$.
	\end{proof}
	
	We will now show that Otelbaev's function has controlled variation.
	\begin{lemma}\label{equiv_q*}
		Let $\mu$ be a non-negative Radon measure on $\mathbb R$. Let $\alpha>0$ and $x\in \mathbb{R}$. Then, for all points $y \in \Delta_x(d_\alpha(x))$ and $z \in   \Delta_x\left(d_\alpha(x)/2\right)$, the following inequalities hold true: 
		\begin{equation*}
			\frac{1}{4}q^*_\alpha(z) \leq q^*_\alpha(x) \leq 4q^*_\alpha(y).
		\end{equation*}
	\end{lemma}
	\begin{proof}
		First, we prove the estimate on the right-hand side. Let $y \in \Delta_x(d_\alpha(x))$. Then, $\Delta_x(d_\alpha(x)) \subset \Delta_y(2d_\alpha(x))$. Hence, by Lemma \ref{max_interval},
		\begin{equation*}
			\mu\left(\Delta_y(2d_\alpha(x))\right) \geq \mu\left(\Delta_x(d_\alpha(x))\right) \geq \frac{1}{\alpha d_\alpha(x)} \geq \frac{1}{2\alpha d_\alpha(x)}.
		\end{equation*}
		Therefore, by the definition of $q_\alpha^*$,
		\begin{equation*}
			q_\alpha^*(y) \geq \frac{1}{4d_\alpha^2(x)} = \frac{1}{4}q^*_\alpha(x).
		\end{equation*}
		To show the estimate on the left-hand side, we let $ z \in \Delta_x\left(d_\alpha(x)/2\right)$. For sufficiently small $\varepsilon>0$, the following inclusion holds true:
		\begin{equation*}
			\Delta_z\left(\frac{d_\alpha(x) - 2\varepsilon}{2}\right)\subset
			\Delta_x(d_\alpha(x) - \varepsilon).
		\end{equation*}
		Therefore, by Lemma \ref{max_interval}, 
		\begin{equation*}
			\mu\left( \Delta_z\left(\frac{d_\alpha(x) - 2\varepsilon}{2}\right) \right)
			\leq \mu\left( \Delta_x(d_\alpha(x) - \varepsilon) \right) < \frac{1}{\alpha(d_\alpha(x) - \varepsilon)} < \frac{1}{\alpha\left(\frac{d_\alpha(x) - 2\varepsilon}{2}\right)}.
		\end{equation*}
		Since, $q_\alpha^*(z) \leq \frac{4}{(d_\alpha(x) - 2\varepsilon)^2}$ for all sufficiently small $\varepsilon>0$, by the definition of $q_\alpha^*$,  we arrive at
		$  q_\alpha^*(z) \leq \frac{4}{d_\alpha^2(x)} = 4q_\alpha^*(x)$,
		just what we need. 
	\end{proof}

	Next, we give a reformulation of the Brinck condition \eqref{Brinck} in terms of $q_\alpha^*$:
	\begin{lemma}\label{Brinc_it_terms_q*}
		Let $\mu$ be a non-negative Radon measure. The following statements are equivalent:
		\begin{enumerate}
			\item \label{semi_boud_q_2}For any $\alpha>0$, $q^*_\alpha$ is a bounded function.
			\item \label{semi_boud_mesure_2} The  condition \eqref{Brinck} is satisfied.
		\end{enumerate}
	\end{lemma}
	\begin{proof}
		If \eqref{semi_boud_q_2} does not hold, there exists  a sequence $\{x_k\}$ such that $q_\alpha^*(x_k)>k$, or equivalently, $d_\alpha(x_k)< 1/\sqrt{k}$. By Lemma \ref{max_interval},
		\begin{equation*}
			\mu\left(\Delta_{x_k}(1)\right) \geq \mu\left(\Delta_{x_k}(d_\alpha(x_k))\right) > \frac{1}{\alpha d_\alpha(x_k)} > \frac{\sqrt{k}}{\alpha},
		\end{equation*}
		and hence, \eqref{semi_boud_mesure_2} is violated.
		
		On the opposite, if  \eqref{semi_boud_q_2} holds,  so that  $q_\alpha^*(x)<C$ for all $x\in \mathbb{R}$, it follows, $d_\alpha(x) > 1/\sqrt{C}$, and Lemma \ref{max_interval} implies
		\begin{equation*}
			\mu\left(\Delta_x\left(1/\sqrt{C}\right)\right)<\frac{\sqrt{C}}{\alpha}.
		\end{equation*}
		Therefore,
		\begin{equation*}
			\mu([x,x+1]) < ([\sqrt{C}]+1) \frac{\sqrt{C}}{\alpha},
		\end{equation*}
		which gives \eqref{semi_boud_mesure_2}.  
	\end{proof}
	
	From Lemmas \ref{semiboundedness} and \ref{Brinc_it_terms_q*}, we derive the following criterion for semi-boun\-ded\-ness of the quadratic form  $\mathbf{a}_\mu$ in terms of $q_\alpha^*$.
	
	\begin{corollary}
		Let $\mu$ be a non-negative Radon measure. Then the sesquilinear form $\mathbf{a}_\mu$ is semi-bounded if and only if for any fixed $\alpha>0$, the function $q_\alpha^*$ is bounded. 
	\end{corollary}

	Next, we study the behavior of the Otelbaev function at infinity.
	\begin{lemma}\label{q*_and_int}
		For any $\alpha>0,$ the following statements are equivalent:
		\begin{enumerate}
			\item $q^*_\alpha(x) \rightarrow 0$ as $x\rightarrow \infty$; \label{itemone_q*_and_int}
			\item \label{itemtwo_q*_and_int} For any fixed $d>0$, $\mu\left(\Delta_x(d)\right) \rightarrow 0$ as $x\rightarrow \infty$.
		\end{enumerate}
		The second property is just a re-formulation of \eqref{MazCond}. Therefore, both properties above, by Theorem \ref{Maz}, are equivalent to the discreteness of the negative spectrum of the operator $\mathbf{H}=-\frac{d^2}{dx^2}-\mu$.
	\end{lemma}
	\begin{proof}
		Assume that \eqref{itemone_q*_and_int} does not hold. Then, $q^*_\alpha(x_k)>C$ for some  sequence $\{x_k\}$ tending to infinity and some $C>0$. Therefore, $d_\alpha(x_k)< 1/\sqrt{C}$. Now, by Lemma \ref{max_interval}, 
		\begin{equation*}
			\mu\left(\Delta_{x_k}\left(1/\sqrt{C}\right)\right) \geq \mu\left(\Delta_{x_k}(d_\alpha(x_k))\right) \geq \frac{1}{\alpha d_\alpha(x_k)} \geq \frac{\sqrt{C}}{\alpha},
		\end{equation*}
		which contradicts $(2)$.
		
		Conversely, if \eqref{itemone_q*_and_int} holds but \eqref{itemtwo_q*_and_int} does not,  there exist a sequence $\{x_k\}$ tending to infinity and constants $d$, $C>0$ such that
		\begin{equation}\label{intq_large_C}
			\mu\left(\Delta_{x_k}(d)\right) > C.
		\end{equation}
		We choose $\varepsilon>0$ such that 
		\begin{equation}\label{d_epsilon_C_relation}
			\frac{d}{2} < \frac{1}{2\sqrt{\varepsilon}}, \qquad \frac{\sqrt{\varepsilon}}{\alpha} < C.
		\end{equation}
		Since \eqref{itemone_q*_and_int} holds, there exists $N>0$ such that $q_\alpha^*(x)<\varepsilon$ for all $|x| >N$. Therefore, $d_\alpha(x)> 1/\sqrt{\varepsilon}$ for all $|x| >N$. Relation \eqref{d_epsilon_C_relation} and the second part of Theorem \ref{max_interval} imply now
		\begin{equation*}
			\mu\left(\Delta_{x_k}(d)\right) \leq \mu\left(\Delta_{x_k}\left(1/\sqrt{\varepsilon}\right)\right) \leq \frac{\sqrt{\varepsilon}}{\alpha} < C,
		\end{equation*}
		for sufficiently large $k$, such  that $|x_k| > N$. This contradicts \eqref{intq_large_C}.
	\end{proof}
	For different values of the averaging parameter,  Otelbaev's functions are equivalent:
	
	\begin{lemma}\label{lem_qalpha_qbeta_equiv}
		Let $\mu$ be a non-negative Radon measure. If $0<\beta<\alpha$, then
		\begin{equation*}
			q_\beta^*(x) \leq q_\alpha^*(x) \leq \frac{\alpha^2}{\beta^2} q_\beta^*(x),
		\end{equation*}
		for all $x\in \mathbb{R}$.
	\end{lemma}
	
	\begin{proof}
		For a fixed $x\in\mathbb{R}$, it is clear that $\alpha\mapsto q_\alpha^*(x)$ is an increasing function, so that $q_\beta^*(x) \leq q_\alpha^*(x)$. If $q_\beta^*(x) = q_\alpha^*(x)$,  the statement is obviously true. Otherwise, if $q_\beta^*(x) < q_\alpha^*(x)$, then $d_\alpha(x) < d_\beta(x)$, this means that  $d_\alpha(x) < d_\beta(x) -\varepsilon$ for some sufficiently small $\varepsilon>0$. Therefore, by Lemma \ref{max_interval},
		\begin{equation*}
			\frac{1}{\alpha d_\alpha(x)} \leq \mu(\Delta_x(d_\alpha(x))) \leq \mu(\Delta_x(d_\beta(x) - \varepsilon)) \leq \frac{1}{\beta (d_\beta(x) - \varepsilon)},
		\end{equation*}
		for any sufficiently small $\varepsilon>0$. Thus,
		\begin{equation*}
			\frac{1}{ d_\alpha(x)} \leq \frac{\alpha}{\beta d_\beta(x)}.
		\end{equation*}
		Recalling the relation between $q_\alpha^*$ and $d_\alpha$ completes the proof.
	\end{proof}

	The following lemma concerns the dependence of $q_\alpha^*$ on the choice of the measure.
	
	\begin{lemma}\label{lem_est_lin}
		Let $\mu_1$ and $\mu_2$ be non-negative Radon measures and $\mu=\mu_1+\mu_2$. Then, for any $\alpha>0$, 
		\begin{equation}\label{est_lin}
			q_{\alpha,\mu}^*(x) \leq 2 \left(q_{\alpha,\mu_1}^*(x) + q_{\alpha,\mu_2}^*(x)\right).
		\end{equation}
	\end{lemma}
	\begin{proof}
		By Lemma \ref{max_interval},
		\begin{equation}
			\mu_j(\Delta_x(d_{\alpha,\mu_j}(x))) \geq \frac{1}{\alpha d_{\alpha,\mu_j}(x)},
			\quad
			\text{for } x\in \mathbb{R}, \; j=1,2.
		\end{equation}
		We note that 
		\begin{equation}\label{mu 12}
			d_{\alpha,\mu_j}(x) \geq d_{\alpha,\mu}(x)
			\quad
			\text{for } j=1,2.
		\end{equation}
		If $d_{\alpha,\mu_1}(x)  = d_{\alpha,\mu}(x) $, then $q_{\alpha,\mu_1}^*(x) = q_{\alpha,\mu}^*(x)$ and \eqref{est_lin} holds. If the estimate in \eqref{mu 12} is strict for both $j=1,2$, i.e., for some $\varepsilon>0$, 
		\begin{equation*}
			d_{\alpha,\mu_j}(x) - \varepsilon> d_{\alpha,\mu}(x),
		\end{equation*}
		then, by Lemma \ref{max_interval},
		\begin{multline*}
			\frac{1}{\alpha (d_{\alpha,\mu_1}(x) - \varepsilon)} + \frac{1}{\alpha (d_{\alpha,\mu_2}(x) - \varepsilon)} > \mu_1(\Delta_x(d_{\alpha,\mu_1}(x) - \varepsilon)) + \mu_2(\Delta_x(d_{\alpha,\mu_2}(x) - \varepsilon))\\
			\geq \mu_1(\Delta_x(d_{\alpha,\mu}(x))) + \mu_2(\Delta_x(d_{\alpha,\mu}(x))) = \mu(\Delta_x(d_{\alpha,\mu}(x))) \geq  \frac{1}{\alpha d_{\alpha,\mu}(x) }
		\end{multline*}
		holds true for any sufficiently small $\varepsilon>0$. Now \eqref{est_lin} follows from
		\begin{equation*}
			\frac{1}{\alpha (d_{\alpha,\mu_1}(x) )} + \frac{1}{\alpha (d_{\alpha,\mu_2}(x) )} \geq \frac{1}{\alpha d_{\alpha,\mu}(x) }.
		\end{equation*}
	\end{proof}
	
	We conclude this section by showing that Otelbaev's function satisfies a certain lower bound for its decay at infinity:
	\begin{proposition}
		Let $0 \neq \mu$ be a non-negative Radon measure. Then, 
		\begin{align*} 
			\liminf_{|x| \to \infty} q_\alpha^\ast(x) |x|^2 \geq \frac{1}{4}.
		\end{align*}
	\end{proposition}
	\begin{proof}
		We first assume that $\mu$ is supported in some interval $[-R, R]$. Further, we only consider $x > 0$, since the case $x < 0$ can be dealt with analogously. 
		
		Since $\mu$ is a compactly supported Radon measure, it is finite. Assume that $x$ is large enough such that $2(x+R) > \frac{\mu(\mathbb R)}{\alpha}$. Then, with $d = 2(x+R)$, we obtain
		\begin{align*}
			\mu(\Delta_x(d)) = \mu([-R, 2x+R]) = \mu(\mathbb R) > \frac{1}{2\alpha(x+R)} = \frac{1}{\alpha d}.
		\end{align*}
		Hence, we see that for every $x$ sufficiently large,
		\begin{align*}
			q_\alpha^\ast(x) \geq \frac{1}{d^2} = \frac{1}{4(x+R)^2},
		\end{align*}
		which implies
		\begin{align*}
			\liminf_{x \to \infty} q_\alpha^\ast(x) \cdot x^2 \geq \liminf_{x \to \infty} \frac{x^2}{4(x+R)^2} = \frac{1}{4}.
		\end{align*}
		If $\mu$ is not compactly supported, then for arbitrary $R > 0$ we have $\mu \geq \mathbf 1_{[-R, R]} \mu$, so
		\begin{align*}
			q_{\mu, \alpha}^\ast(x) \geq q_{\mathbf 1_{[-R, R]}\mu, \alpha}^\ast(x),
		\end{align*}
		and therefore
		\begin{align*}
			\liminf_{x \to \infty} q_{\mu, \alpha}^\ast(x) \cdot x^2 \geq q_{\mathbf 1_{[-R, R]}\mu, \alpha}^\ast(x) \cdot x^2 \geq \frac{1}{4}.
		\end{align*}The estimate at $-\infty$ is dealt with analogously. This finishes the proof.
	\end{proof}
	\begin{remark}
		If $\mu$ is compactly supported, then we also have 
		\begin{align*}
			\limsup_{|x|\to \infty} q_\alpha^\ast(x) \cdot x^2 \leq 1,
		\end{align*}
		which will be clear from the discussion in Example \ref{ex:compsupp}.    
	\end{remark}

	\section{Eigenvalue estimates}
	
	Lemma \ref{q*_and_int} gave us a condition of discreteness of the negative spectrum in terms of the function $q_\alpha^*$. In terms of the same function, we now find 
	two-sided estimates for the eigenvalues of $\Hb_\mu$. The reasoning uses the standard bracketing idea, in the realization of,  e.g., \cite{Weidl}, adapted to general measures.

	\subsection{Auxiliary lemmas}\label{sec:decomposition} 
	
	We consider a covering of $\R$ by closed intervals which have common endpoints. If such endpoint coincides with the position of a point mass of the measure $\mu$ we split this point mass in two, in a special way, and assign each summand to the corresponding interval. Note that besides depending on $\mu$, the construction also depends on $\alpha > 0$. Further, we always assume that $\mu$ satisfies Brinck's condition \eqref{Brinck}.
	
	We start with $a_0 = 0$ and $\gamma_0 = 0$. If $\mu([a_0,+\infty)) = 0$, then we define $a_{1}= +\infty$. Otherwise, we set 
	\begin{equation*}
		a_{1} := \sup\left\{x\in (a_0,+\infty): \ \mu([a_0, x)) \leq \frac{1}{\alpha|x - a_0|} \right\}.
	\end{equation*}
	Since 
	\begin{align*}
		\mu([a_0, x)) \uparrow \mu([a_0, \infty)) > 0 \quad \text{ and } \quad \frac{1}{\alpha|x-a_0|} \downarrow 0 \quad \text{as }  x \to \infty,
	\end{align*} 
	$a_1$ is well-defined and finite. Further, since both $\mu([a_0, x)) \rightarrow \mu(\{ a_0\}) < \infty$ and $(\alpha|x-a_0|)^{-1} \rightarrow \infty$ continuously as $x \to 0$, we also obtain $a_1 > a_0$.
	
	Having constructed $a_1$ in this way, we have the following fact:
	\begin{lemma}\label{lem:gamma}
		If $a_1 < \infty$, then $\mu([a_0, a_1]) \geq \frac{1}{\alpha (a_1 - a_0)}$.
	\end{lemma}    
	\begin{proof}
		Assume it were $\mu([a_0, a_1]) < \frac{1}{\alpha (a_1 - a_0)}$. Since
		\begin{align*}
			\mu([a_0, a_1 + \varepsilon)) \overset{\varepsilon \to 0}{\longrightarrow} \mu([a_0, a_1]),
		\end{align*}
		we would have for some $\varepsilon' > 0$ sufficiently small:
		\begin{align*}
			\mu([a_0, a_1 + \varepsilon')) \leq \frac{1}{\alpha (a_1 + \varepsilon' - a_0)}.
		\end{align*}
		But this contradicts the definition of $a_1$.
	\end{proof}
	If $a_1 = \infty$, we set $I_0 = [a_0, \infty)$, define the measure $\mu_0$ on $I_0$ as the zero measure and end the construction here. If $a_1 < \infty$, we can set (based on the previous lemma) $\gamma_1 = \frac{1}{\alpha (a_1 - a_0)} - \mu([a_0, a_1)) \geq 0$. We define the interval $I_0$ as $I_0 = [a_0, a_1]$ and the measure $\mu_0$ on $I_0$ as $\mu_0 = \mu \cdot \mathbf 1_{[a_0, a_1)} + \gamma_1 \delta_{a_1}$. We now continue the construction inductively.
	After having constructed \emph{finitely} many points $a_0,\cdots,a_k$, intervals $I_0, \dots, I_{k-1}$, constants $\gamma_1, \dots, \gamma_k$ and measures $\mu_0, \dots, \mu_{k-1}$, we define $a_{k+1}$ analogously, that is: if $\mu([a_k,+\infty)) - \gamma_k= 0$ we set $a_{k+1}= +\infty$ and $I_k = [a_k, \infty)$, $\mu_k = 0$ on $I_k$. Otherwise
	\begin{equation*}
		a_{k+1} := \sup\left\{x\in (a_k,+\infty): \ \mu([a_k, x)) - \gamma_k \leq \frac{1}{\alpha|x - a_k|} \right\}.
	\end{equation*}
	As above, $a_{k+1}$ is well-defined. If $a_{k+1} = \infty$, we let $I_k = [a_k, \infty)$, $\mu_k = 0$ on $I_k$ and end the construction here. Otherwise, analogously to Lemma \ref{lem:gamma}, we can set $\gamma_{k+1} = \frac{1}{\alpha (a_{k+1} - a_k)} - \mu([a_k, a_{k+1})) + \gamma_k \geq 0$, $I_k = [a_k, a_{k+1}]$, $\mu_k$ is the measure on $I_k$ given by $\mu_k = \mu \cdot \mathbf 1_{[a_k, a_{k+1})} - \gamma_k \delta_{a_k} + \gamma_{k+1} \delta_{a_{k+1}}$.
	
	For $k \leq 0$, we repeat the construction symmetrically, i.e., we construct the points
	\begin{align*}
		a_{k-1} := \inf \left\{ x \in (-\infty, a_k): ~\mu((x, a_k]) - \gamma_k \leq \frac{1}{\alpha (a_k - x)} \right\} 
	\end{align*}    
	and for $a_{k-1} > -\infty$ we set $I_{k-1} = [a_{k-1}, a_k]$, $\gamma_{k-1} = \frac{1}{\alpha (a_k - x)} - \mu((a_{k-1}, a_k]) + \gamma_k$. We denote by $\K\subset\Z$ the set of ordinals $k$ of these points (with exception of $+\infty$ if such point is present). If $a_k \neq -\infty$ and $a_{k+1} \neq \infty$, we denote by $x_k$ the mid-point of the interval $I_k$. With these definitions, the intervals $I_k$ form a partition of $\mathbb R$. We emphasize that this whole construction implicitly depends on the choice of the measure $\mu$ and the constant $\alpha > 0$, even though we do not capture this in the notation of the intervals and the points.
	\begin{remark}
		We want to emphasize that the assumption of Brinck's condition \eqref{Brinck} on $\mu$ enforces that $a_k$ cannot converge to some finite point as $k \to \pm \infty$, i.e., the intervals $I_k$ indeed form a covering of $\mathbb R$.
	\end{remark}
	We have now constructed a covering $\{ I_k\}_{k \in \mathbb K}$ of $\mathbb R$ and a collection of measures $\mu_k$ such that for every Borel set $A \subset \mathbb R$: $\mu(A) = \sum_{k \in \mathbb K} \mu_k(A_k \cap I_k)$. With this data, we will perform our Neumann bracketing later.

	\begin{lemma}\label{lemma:decomp}
		The decomposition has the following properties:
		\begin{enumerate}
			\item \label{lemma:decomp_1}For $k \in \mathbb K$ with $\mu_k(I_k) > 0$: $\mu_k(I_k) |I_k| = \frac{1}{\alpha}$.
			\item \label{lemma:decomp_2} It is: $d_\alpha(x_k) = |I_k|$. 
		\end{enumerate}
	\end{lemma}
	\begin{proof}
		The first part is satisfied by construction. To show the second part, we assume that $|I_k| < d_\alpha(x_k)$, i.e., $a_{k+1} - a_k < d_\alpha(x_k)$. Then,
		\begin{equation*}
			x_k - \frac{d_\alpha(x_k) - \varepsilon}{2} < a_k < a_{k+1} < x_k + \frac{d_\alpha(x_k) - \varepsilon}{2}
		\end{equation*}
		for all sufficiently small  $\varepsilon > 0$. By construction of $a_{k+1}$, we have
		\begin{align*}
			\mu \left( \left [ a_k, x_k + \frac{d_\alpha(x_k) - \varepsilon}{2} \right) \right) > \frac{1}{\alpha \left( x_k + \frac{d_\alpha(x_k) - \varepsilon}{2} - a_k\right) } + \gamma_k \geq \frac{1}{\alpha \left( x_k + \frac{d_\alpha(x_k) - \varepsilon}{2} - a_k \right) }. 
		\end{align*}
		On the other hand, by Lemma \ref{max_interval},
		\begin{equation*}
			\mu\left( \left[a_k,x_k + \frac{d_\alpha(x_k) - \varepsilon}{2}\right) \right) \leq \mu\left( \Delta_{x_k}(d_\alpha(x_k) - \varepsilon) \right)
			< \frac{1}{\alpha (d_\alpha(x_k) - \varepsilon)}.
		\end{equation*}
		Therefore,
		\begin{equation*}
			\frac{1}{\alpha\left(x_k + \frac{d_\alpha(x_k) - \varepsilon}{2} - a_k\right)} < \frac{1}{\alpha (d_\alpha(x_k) - \varepsilon)}
		\end{equation*}
		for sufficiently small $\varepsilon > 0$. This is equivalent to
		\begin{align*}
			d_\alpha(x_k) - \varepsilon &< x_k + \frac{d_\alpha(x_k) - \varepsilon}{2} - a_k \\
			&= \frac{a_{k+1} + a_k}{2} + \frac{d_\alpha(x_k) - \varepsilon}{2} - a_k
			=\frac{a_{k+1} - a_k}{2} + \frac{d_\alpha(x_k) - \varepsilon}{2},
		\end{align*}
		and  this contradicts  our initial assumption that $d_\alpha(x_k) > a_{k+1} - a_k$.
		
		Now, assume that $d_\alpha(x_k) < |I_k|$. Then,
		\begin{equation*}
			\Delta_{x_k}(d_\alpha(x_k)) \subset I_k^{int},
		\end{equation*}
		where $I_k^{int}$ is the interior of $I_k$.  This and Lemma \ref{max_interval} imply
		\begin{equation*}
			\mu_k(I_k) \geq \mu(I_k^{int}) \geq \mu(\Delta_{x_k}(d_\alpha(x_k))) > \frac{1}{\alpha d_\alpha(x_k)} > \frac{1}{\alpha|I_k|}.
		\end{equation*}
		This contradicts part \eqref{lemma:decomp_1} of the present lemma.\qedhere
	\end{proof}

	\subsection{Neumann bracketing}\label{Subsect.Bracketing}
	We consider here   Sturm-Liouville operators with Neumann boundary conditions on each interval $I_k$, with the potentials being the measure $\mu_k$. To define these operators precisely, we need the following lemmas.
	
	\begin{lemma}{\cite[Lemma 2]{Brinck}}\label{lemma_Brinck}
		Let $I \subset \mathbb R$ be an interval of length $d$ and $f \in H^1(I)$. Then, for all $y \in I$,
		\[ \frac{1}{2d} \int_{I}|f(x)|^2  dx - \frac{d}{2}\int_{I}|f'(x)|^2  dx \leq |f(y)|^2 \leq \frac{2}{d} \int_{I}|f(x)|^2  dx + d \int_{I}|f'(x)|^2  dx.\]
	\end{lemma}
	
	\begin{lemma}
		Let $\alpha>0$ and $I_k$ be some interval as defined in Section \ref{sec:decomposition}. Then, the sesquilinear form
		\begin{equation}\label{sesq_form_k}
			\mathbf{a}_\mu^k[f]=\int_{I_k}|f'|^2 dx - \int_{I_k}|f|^2 d\mu_k, \qquad 
			\mathfrak{d}(\mathbf{a}_\mu^k)= H^1(I_k),
		\end{equation}
		is lower semi-bounded, closed, symmetric, and densely defined in $L^2(I_k)$. 
	\end{lemma}
	\begin{proof}
		If $\mu_k(I_k) = 0$, nothing needs to be proven. Otherwise, it is clear that $\mathbf{a}_\mu^k$ is symmetric and densely defined. Due to Lemma \ref{lemma:decomp}, $|I_k| \mu_k(I_k) = \frac{1}{\alpha}$. Hence, by Lemma \ref{lemma_Brinck}, for $f\in H^1(I_k)$ we obtain:
		\begin{align*}
			\left|\int_{I_k} |f|^2d\mu_k\right| &\leq \frac{2\mu_k(I_k)}{|I_k|}  \|f\|_{L^2(I_k)}^2 + |I_k|\mu_k(I_k)  \|f'\|_{L^2(I_k)}^2\\
			&\leq \frac{2\mu_k(I_k)}{|I_k|} \| f\|^2_{L^2(I_k)} + \frac{1}{\alpha + 1} \cdot (1 + \frac{1}{\alpha}) \| f'\|_{L^2(I_k)}^2.
		\end{align*}
		Since $(1+\alpha^{-1})\| f'\|_{L^2(I_k)}^2$ is a non-negative, closed, symmetric and densely defined quadratic form on the domain $H^1(I_k)$ and $\frac{1}{1+\alpha} < 1$, the KLMN-Theorem implies that $\mathbf{a}_\mu^k$ is closed and semi-bounded; see, e.g., \cite[Theorem X.17]{Reed2}.
	\end{proof}
	
	Our upper estimates for eigenvalues  will be, as this is usually done, by means of the Neumann bracketing. The standard application of the variational principle (say, the Glazman lemma) is the following.
	\begin{lemma}\label{ineq_for_N}
		Let $\mu$ be a Radon measure satisfying \eqref{Brinck} and $\Hb_\mu$ be the operator generated by the sesquilinear form \eqref{form}. Let $\alpha>0$ and $\{I_k\}_{k\in \mathbb{K}}$ be the intervals introduced above. Consider the operators $\mathbf{H}_\mu^k$ associated with the sesquilinear forms \eqref{sesq_form_k}. Then, for every $\lambda >0$, the following inequality holds true:
		\begin{equation*}
			N(-\lambda,\Hb_\mu)\leq \sum_{k \in \mathbb K} N(-\lambda,\Hb_\mu^k).
		\end{equation*}
	\end{lemma}
	
	In our first estimate, we use the intervals $I_k$ corresponding to the special value $\alpha=2$.
	\begin{lemma}\label{N_leq_sum_N_k}
		Under the conditions above, for   $\alpha = 2$, the following estimate holds for $\lambda>0$:
		\begin{equation*}
			N(-\lambda, \Hb_\mu) \leq  \sum_{k \in \mathbb K: ~-\lambda \geq -\frac{1}{|I_k|^2}} 1 = \#\{ k\in \mathbb{K}: \lambda\leq 1 /|I_k|^2\}.
		\end{equation*}
	\end{lemma}
	\begin{proof}
		Let $\lambda>0$. Let $\Hb_\mu^k$ be the operators from Lemma \ref{ineq_for_N} and $\mathbf{a}_\mu^k$ be the corresponding quadratic forms. Then, 
		\begin{equation}\label{N_leq_sum_N_k_eq}
			N(-\lambda, \Hb_\mu) \leq  \sum_{k\in \mathbb{K}} N(-\lambda, \Hb_\mu^k).
		\end{equation}
		By construction, if some $I_k$ is unbounded, then it follows that $\mu_k(I_k) = 0$, and hence $N(-\lambda,H_\mu^k) = 0$. Therefore, from now on, we will only consider bounded intervals $I_k$. By Lemma \ref{lemma_Brinck},
		\begin{equation*}
			\mathbf{a}_\mu^k[f] \geq \left( 1 - |I_k|\mu_k(I_k) \right) \|f'\|_{L^2(I_k)}^2 - \frac{2}{|I_k|} \mu_k(I_k) \|f\|_{L^2(I_k)}^2.
		\end{equation*}
		From Lemma \ref{lemma:decomp} we know that $2\mu(I_k) = 1/|I_k|$. Hence, the estimate above gives 
		\begin{equation}\label{estimate_for_a_k}
			\mathbf{a}_\mu^k[f] \geq \mathbf{b}^k[f], 
		\end{equation}
		where $\mathbf{b}^k$ is the sesquilinear form defined as follows:
		\begin{equation*}
			\mathbf{b}^k[f] = \frac{1}{2} \|f'\|_{L^2(I_k)}^2 - \frac{1}{|I_k|^2} \|f\|_{L^2(I_k)}^2, \qquad \mathfrak{d}(\mathbf b^k) = H^1(I_k), \, f\in H^1(I_k).
		\end{equation*}
		Therefore, 
		\begin{equation}\label{HgeqL}
			\Hb_\mu^k\geq \Bb^k,
		\end{equation}
		in the sense of quadratic forms, where $\Bb^k$ is the operator associated with $\mathbf{b}^k$. Note that $\Bb^k$ is, up to a constant factor, $(-\frac{d^2}{dx^2})_{\Nc} - c I$, with $(-\frac{d^2}{dx^2})_\Nc$ being the Neumann Laplacian of the interval $I_k$. Hence, the spectrum of $\Bb^k$ is well-known: it is purely discrete and consists of the eigenvalues
		\begin{equation*}
			\eta_m =  \frac{1}{2}\left(\frac{\pi m}{|I_k|}\right)^2 - \frac{1}{|I_k|^2},\quad m\in \mathbb{N} \cup \{0\}.
		\end{equation*}
		For the lowest eigenvalue of $\Bb^k$,
		we have $\eta_0 = -1/|I_k|^2$, but the next eigenvalue is already larger than $-\lambda$:
		\begin{equation*}
			\eta_1 = \left(\frac{\pi^2}{2} - 1\right)\frac{1}{|I_k|^2} > 0.
		\end{equation*}
		Therefore,
		\begin{equation*}
			N(-\lambda,\Bb^k) = 
			\begin{cases}
				1 & \text{if }  -1/|I_k|^2<-\lambda,\\
				0 & \text{otherwise},
			\end{cases}
		\end{equation*}
		and hence, \eqref{HgeqL} implies that $N(-\lambda,\Hb_\mu^k) \leq 1$. 
		Moreover, from \eqref{estimate_for_a_k}, we know that 
		\begin{equation*}
			N(-\lambda,\Hb_\mu^k) = 0, \quad \text{for } -\lambda< -\frac{1}{|I_k|^2}.
		\end{equation*}
		This all means that the intervals $I_k$ with $-1/|I_k|^2<-\lambda$ contribute with at most one eigenvalue to the sum \eqref{N_leq_sum_N_k_eq}, while the remaining intervals do not make any contribution to this sum at all.
	\end{proof}

	\subsection{Two-sided estimate for the spectral distribution function.}
	Now, we are ready to obtain the main results of this section, which are expressed in terms of the measure of sublevel sets for $q_\alpha^*$:
	\begin{equation*}
		M_\alpha(\lambda) : = \{x\in \mathbb{R}: \; q_\alpha^*(x) \geq \lambda\}.
	\end{equation*}
	
	\begin{theorem}\label{est_for_N_supnorm}
		Let $\mu$ be a Radon measure satisfying \eqref{Brinck} and $\Hb_\mu$ be the operator generated by sesquilinear form \eqref{form}. Let
		\begin{equation*}
			\alpha = \left(\pi^2 + \frac{3}{4}\right)^{-1}, \qquad \beta =2.
		\end{equation*}
		Then, for $\lambda>0$ the following estimates hold true:
		\begin{align}
			\frac{1}{2} \min_{y\in M_\alpha(4\lambda)} \sqrt{q^*_{\alpha}(y)}  \mathcal{L} (M_\alpha(4\lambda)) &\leq N(-\lambda, \Hb_\mu) \leq \max_{y\in M_\beta(\lambda/4)} \sqrt{q^*_{\beta}(y)} \cdot \mathcal{L} (M_\beta(\lambda/4)),\\
			\sqrt{\lambda} \mathcal{L}(M_\alpha(4\lambda)) &\leq N(-\lambda, \Hb_\mu) 
			\leq \sqrt{\|q_\beta^*\|_{\infty}} \mathcal{L} (M_\beta(\lambda/4))\label{thm:estimate_eq2},\\
			\sup_{\sqrt{\eta}\geq \lambda} \left(\eta \cdot \mathcal{L}(M_\alpha(4\eta))\right)&\leq N(-\lambda, \Hb_\mu) 
			\leq 4 \int_{M_\beta(\lambda/4)} \sqrt{q_\beta^*(y)}dy.
		\end{align}
		Here, $\mathcal{L}$ is the Lebesgue measure.
	\end{theorem}
	
	\begin{remark}\label{finiteness.rem} Theorem \ref{est_for_N_supnorm} does not make \emph{a-priori}  assumptions about the nature of the negative spectrum of $\Hb_{\mu}$. The inequalities should be understood in the following way. If, for a given $\lambda>0$, one of the quantities on the right-hand side is finite, then the spectrum below $-\lambda$ is finite, and the inequalities hold. On the other hand, if some of the quantities on the left-hand side is infinite, the spectrum of $\Hb_\mu$ in $(-\infty, -\lambda]$ is infinite and thus contains at least one point of the essential spectrum.   
	\end{remark}

	\begin{proof}
		\textbf{The upper bound.} We apply Lemma \ref{N_leq_sum_N_k}, which  gives
		\begin{equation}\label{N_less_summ_1}
			N(-\lambda, \Hb_\mu) \leq  \sum_{k \in \mathbb K:~ -\lambda \geq -\frac{1}{|I_k|^2}} 1,
		\end{equation}
		where the $I_k$ are the intervals defined as Section \ref{sec:decomposition} with $\alpha = 2$. For those $k\in \mathbb{K}$ for which
		\begin{equation}\label{spec_k}
			-\lambda \geq -\frac{1}{|I_k|^2},
		\end{equation}
		we know from Lemma \ref{lemma:decomp} that $I_k = \Delta_{x_k}(d_2(x_k))$. 
		Therefore, for any $y \in I_k$, Lemma \ref{equiv_q*} and \eqref{spec_k} imply that
		\begin{equation*}
			q_2^*(y)\geq \frac{1}{4}q_2^*(x_k) = \frac{1}{4|I_k|^2}\geq \frac{\lambda}{4}.
		\end{equation*}
		Hence, we obtain
		\begin{equation*}
			1 = |I_k| \cdot \frac{1}{|I_k|} \leq \frac{1}{|I_k|} \cdot \mathcal{L} \left( \left\{x\in I_k: q_2^*(x)\geq \frac{\lambda}{4} \right\} \right) .
		\end{equation*}
		Then,
		\begin{align*}
			N(-\lambda, \Hb_\mu) &\leq \sum_{k \in \mathbb K:~-\lambda \geq -\frac{1}{|I_k|^2}} \frac{1}{|I_k|} \cdot \mathcal{L} \left( \left\{x\in I_k: q_2^*(x)\geq \frac{\lambda}{4} \right\} \right) \\
			&\leq \max_{y\in \left\{x\in I_k: q_2^*(x)\geq \frac{\lambda}{4} \right\}} \sqrt{q_2^*(y)} \cdot \mathcal{L} \left( \left\{x\in \mathbb{R}: q_2^*(x)\geq \frac{\lambda}{4} \right\} \right) .
		\end{align*}
		This gives the first upper bound. The second upper bound is an obvious consequence of the first one. The third upper bound is obtained from:
		\begin{equation*}
			N(-\lambda, \Hb_\mu) \leq   \sum_{-\lambda \geq -\frac{1}{|I_k|^2}} \int_{I_k}\frac{1}{|I_k|} \leq 4\sum_{-\lambda \geq -\frac{1}{|I_k|^2}} \int_{I_k}\sqrt{q^*_{\alpha}} \leq  4 \int_{\left\{x\in\mathbb{R}: q_2^*(x)\geq \frac{\lambda}{4}\right\}} \sqrt{q_2^*(y)}dy.
		\end{equation*}

		\textbf{The lower bound.} For a fixed   $\lambda>0$, we define intervals
		\begin{align*}
			&J_k : = \left( \frac{k-1}{\sqrt{\lambda}}, \frac{k}{\sqrt{\lambda}}\right)\\
			& \tilde{J}_k := \left( \frac{k-1}{\sqrt{\lambda}}+ \frac{1}{4\sqrt{\lambda}}, \frac{k}{\sqrt{\lambda}} - \frac{1}{4\sqrt{\lambda}}\right)\\
			&\tilde{\tilde{J}}_k := \left( \frac{k-1}{\sqrt{\lambda}} + \frac{3}{8\sqrt{\lambda}}, \frac{k}{\sqrt{\lambda}}  - \frac{3}{8\sqrt{\lambda}}\right)
		\end{align*}
		for $k\in \mathbb{Z}$. Note that $J_k$, $\tilde{J}_k$, and $\tilde{\tilde{J}}_k$ have the same centerpoint and their lengths, respectively, are $\frac{1}{\sqrt{\lambda}}$, $\quad \frac{1}{2\sqrt{\lambda}}$, and $\frac{1}{4\sqrt{\lambda}}$. 
		We define the orthogonal system of test functions
		\begin{eqnarray*}
			f_k(x) := \begin{cases}
				1-\cos\left( 2\pi\sqrt{\lambda}\left(x-\frac{k}{\sqrt{\lambda}}\right)\right), ~&x \in J_k,\\
				0, & x \not \in J_k,
			\end{cases}
		\end{eqnarray*}
		for $k\in \mathbb{Z}$.
		Since $f_k\geq 1$ in $\tilde{J}_k$, we can  estimate the value of the quadratic form on $f_k$:
		\begin{equation*}
			\mathbf{a}_\mu[f_k] = 2\pi^2\sqrt{\lambda} - \int_{J_k} f_k^2~d\mu \leq 2\pi^2\sqrt{\lambda} - \mu(\tilde{J}_k).
		\end{equation*}
		We set $\alpha = \left(\pi^2 + \frac{3}{4}\right)^{-1}$ and suppose that there exists $\xi \in \tilde{\tilde{J}}_k$ such that $q_{\alpha}^\ast(\xi) > 4\lambda$. 
		Then, by the definition of the function $q^*_\alpha$, 
		\begin{equation*}
			\mu (\tilde{J}_k) \geq \mu\left(\left[\xi - \frac{1}{4\sqrt{\lambda}},\xi + \frac{1}{4\sqrt{\lambda}} \right]\right)> \frac{\pi^2 + \frac{3}{4}}{\frac{1}{2\sqrt{\lambda}}} = \left(2\pi^2 + \frac{3}{2}\right)\sqrt{\lambda}.
		\end{equation*}
		Therefore, $\mathbf{a}_\mu[f_k] <  -\frac{3}{2}\sqrt{\lambda}$. 
		Since $\|f_k\|^2 = \frac{3}{2\sqrt{\lambda}}$, the last estimate gives $\mathbf{a}_\mu^k[f_k] < - \lambda \|f_k\|_{L^2(\mathbb{R})}^2$.
		Now, the variational principle implies
		\begin{equation*}
			N(-\lambda,\Hb_\mu) \geq \sum_{ k \in \mathbb Z:~\left\{x\in \tilde{\tilde{J}}_k: q^*_{\alpha}(x) \geq 4\lambda\right\} \neq \emptyset} 1  = 4\sqrt{\lambda} \sum_{k \in \mathbb Z:~ \left\{x\in \tilde{\tilde{J}}_k: q^*_{\alpha}(x) \geq 4\lambda\right\} \neq \emptyset} \frac{1}{4\sqrt{\lambda}}.
		\end{equation*}
		For any interval $I \subset \mathbb R$ we denote by $I+d$ the interval $I\subset \mathbb{R}$  shifted by $d$, that is,
		\begin{equation*}
			I+d := \{x+d: x\in I\}.
		\end{equation*}
		By the arguments above, we know that
		\begin{equation}\label{J_k}
			N(-\lambda,\Hb_\mu) \geq \sum_{k \in \mathbb Z:~ \left\{x\in \tilde{\tilde{J}}_k+\frac{j}{4\sqrt{\lambda}}: q^*_{\alpha}(x) \geq 4\lambda\right\} \neq \emptyset} 1  = 4\sqrt{\lambda} \sum_{k \in \mathbb Z:~ \left\{x\in \tilde{\tilde{J}}_k+\frac{j}{4\sqrt{\lambda}}: q^*_{\alpha}(x) \geq 4\lambda\right\} \neq \emptyset} \frac{1}{4\sqrt{\lambda}}
		\end{equation}
		for $j=0,1,2,3$. The intervals $\tilde{\tilde{J}}_k+\frac{j}{4\sqrt{\lambda}}$ cover $\mathbb{R}^1$ with multiplicity 4, therefore, summing \eqref{J_k} over $j$,
		we obtain
		\begin{equation*}
			4N(-\lambda,H_\mu) \geq 4\sqrt{\lambda} \cdot \mathcal{L}\left( \left\{x\in \mathbb{R}: q^*_{\alpha}(x) \geq 4\lambda\right\} \right),
		\end{equation*}
		and this finishes the proof for the lower bound.\qedhere
	\end{proof}
	
	\subsection{Estimates for eigenvalues}
	It is convenient to reformulate the preceding results as estimates for the individual negative eigenvalues of $\Hb_{\mu}$.
	
	\begin{theorem}\label{est_lambda_n}
		Let $\mu$ be a Radon measure which satisfies  \eqref{Brinck} and $\mathbf{H}_\mu$ be the operator generated by sesquilinear form \eqref{form}. Assume that the negative part of the spectrum of $\mathbf{H}_\mu$ is discrete. Denote by $(\lambda_n)_{n \in \mathbb N}$ the non-decreasing sequence of negative eigenvalues (counting multiplicities) with the convention that $\lambda_{m+1} = \lambda_{m+2} = \dots = \Lambda=\inf\sigma_{ess}(\Hb_\mu)$ if there are only $m$  eigenvalues below the lowest point $\Lambda$ of the essential spectrum. Then,  for each $n \in \mathbb N$,
		\begin{align*}
			\sup &\{ - \lambda < 0: ~\mathcal L(\{ x: ~q_\beta^\ast(x) \geq \lambda/4\}) \leq \frac{n-1}{\sqrt{\| q_\beta^\ast\|_\infty}}\} \\
			\leq & \quad \lambda_n \\
			\leq &\quad \inf \{ -\lambda < 0: ~\mathcal L(\{ x: ~q_\alpha^\ast(x) \geq 4\lambda\}) \geq \frac{n}{\sqrt{\lambda}} \}
		\end{align*}
	\end{theorem}
	\begin{proof}
		Suppose that $-\lambda$ satisfies $\mathcal L(\{ x: ~q_\beta^\ast(x) \geq \lambda/4\}) \leq \frac{n-1}{\sqrt{\| q_\beta^\ast\|_\infty}}$. Then, by Theorem \ref{est_for_N_supnorm}, 
		\begin{align*}
			N(-\lambda, \mathbf{H}_\mu) \leq n-1,
		\end{align*}
		this means $\lambda_n \geq -\lambda$. On the other hand, if $-\lambda$ satisfies $\mathcal L(\{ x: ~q_\alpha^\ast(x) \geq 4\lambda\}) \geq \frac{n}{\sqrt{\lambda}}$, then Theorem \ref{est_for_N_supnorm} yields $N(-\lambda, \mathbf{H}_\mu) \geq n$, 
		hence $\lambda_n \leq -\lambda$. 
	\end{proof}
	
	As corollary, we obtain a \emph{lower} estimate for the number $N_-(\mathbf{H}_{\mu})$ of negative eigenvalues of $\mathbf{H}_{\mu}$. 
	\begin{corollary}
		Under the above general conditions for the measure $\mu$, 
		\begin{equation}\label{lower estimate}
			N_-(\mathbf{H}_{\mu})  \geq \sup_{\varepsilon > 0} \frac{\varepsilon^{3/2}}{2} \int_{\{ q_\alpha^\ast \geq \varepsilon\}} \frac{1}{q_\alpha^\ast} dx.
		\end{equation}
	\end{corollary}
	
	\begin{proof}
		We use the conventions from Theorem \ref{est_lambda_n}. 
		Let  $N(0, \mathbf{H}_\mu) = n$. Then, $\lambda_{n+1} = 0$. Hence, for every $\lambda > 0$,
		\begin{equation}
			\frac{n+1}{\sqrt{\lambda}} > \mathcal L(\{ x: ~q_\alpha^\ast(x) \geq 4\lambda\})\ge
			4\lambda \int_{\{ q_\alpha^\ast \geq 4\lambda\}} \frac{1}{q_\alpha^\ast(x)}~dx.
		\end{equation}
		This yields for every $\varepsilon > 0$:
		\begin{align*}
			n+1 > \frac{\varepsilon^{3/2}}{2} \int_{\{ q_\alpha^\ast \geq \varepsilon\}} \frac{1}{q_\alpha^\ast(x)} ~dx,
		\end{align*}
		just what we need.		
	\end{proof}
	In particular, from \eqref{lower estimate}, we immediately obtain the following \emph{necessary} condition for the finiteness of the negative spectrum:
	\begin{corollary}
		Under the above general conditions for the measure $\mu$, if 
		\begin{align*} 
			\sup_{\varepsilon > 0} \varepsilon^{3/2} \int_{\{ q_\alpha^\ast \geq \varepsilon\}} \frac{1}{q_\alpha^\ast}~dx = \infty,
		\end{align*}
		then $N_{-}(\Hb_\mu) = \infty$. 	
	\end{corollary}
	\begin{remark}
		In the results above, we essentially made use of the lower bounds obtained in Theorem \ref{est_for_N_supnorm}. Since $q_\alpha(x)$ decays at infinity not faster than $|x|^{-2}$, the upper bound in this theorem tends to infinity as $\lambda\to 0$, therefore, it is useless for estimating the number of all negative eigenvalues from above.  
	\end{remark}
	We conclude this section with an estimate for the lowest point of the spectrum. 
	\begin{corollary}\label{lowest}
		Let $\mu$ be a Radon measure, as usual, and $\Hb_\mu$ be the operator generated by sesquilinear form \eqref{form}. Let $\lambda_1 = \inf \sigma(\Hb_\mu)$. Then, the following estimates hold true: 
		\begin{equation*}
			-4\|q_2^*\|_{\infty} \leq \lambda_1 \leq -\frac{1}{4} \|q_{(\pi^2 + \frac{3}{4})^{-1}}^*\|_{\infty}.
		\end{equation*}
	\end{corollary}
	\begin{proof}
		Recall that $\lambda_1 = \sup \{ \lambda \in \mathbb R: ~N(\lambda, H_{\mu}) = 0\}$. Assume that $-\lambda < -4\| q_2^\ast\|_\infty$, i.e., $\lambda/4 > \| q_2^\ast\|_\infty$. Then, by Theorem \ref{est_for_N_supnorm},
		\begin{align*}
			N(-\lambda, \Hb_\mu) \leq \sqrt{\| q_2^\ast\|_\infty} \mathcal L( \{ x: ~q_2^\ast(x) \geq \lambda/4\}). 
		\end{align*}
		By our assumption on $\lambda$, $\{ x: ~q_2^\ast(x) \geq \lambda/4\} = \emptyset$, and therefore $N(-\lambda, \Hb_\mu) \leq 0$, as required. 
		
		On the other hand, if $-\lambda > -\| q_{(\pi^2 + 3/4)^{-1}}^\ast\|_\infty/4$, i.e., $4\lambda < \| q_{(\pi^2 + 3/4)^{-1}}^\ast\|_\infty$, then
		\begin{align*}
			N(-\lambda, \Hb_{\mu}) \geq \sqrt{\lambda} \mathcal L(\{ x: ~q_{(\pi^2 + 3/4)^{-1}}^\ast(x) \geq 4\lambda\}).
		\end{align*}
		Since $4\lambda < \| q_{(\pi^2 + 3/4)^{-1}}^\ast\|_\infty$, the set $\{ x: ~q_{(\pi^2 + 3/4)^{-1}}^\ast(x) \geq 4\lambda\}$ is non-empty. Since $q_{(\pi^2 + 3/4)^{-1}}^\ast$ is continuous, the set needs to have a positive Lebesgue measure. Hence, $N(-\lambda, \Hb_{\mu})>0$ and, being integer it is not less than $1$. This means  that $\lambda_1 \leq -\lambda$. 
	\end{proof}
	
	Note that in Corollary \ref{lowest}, it is not supposed that the negative spectrum of $\Hb_\mu$ is discrete or that the lowest point of the spectrum is an eigenvalue. An estimate for the lowest point of the essential spectrum can be obtained as well.
	
	\begin{corollary}	Under the above conditions,  let $\Lambda = \inf \sigma_{ess}(\Hb_\mu)$. We set
		\begin{equation*}
			Q_1 = \limsup_{x \to \pm \infty} q_2^\ast(x),
			\qquad
			Q_2 = \liminf_{x \to \pm \infty} q_\alpha^\ast(x),
		\end{equation*}
		where $\alpha = (\pi^2 + 3/4)^{-1}$. Then,
		\begin{equation*}
			-4Q_1 \leq \Lambda \leq -\frac{1}{4} Q_2.
		\end{equation*}
	\end{corollary}
	\begin{proof}
		Recall that $\Lambda = \sup \{ \lambda \in \mathbb R: ~N(\lambda, \Hb_{\mu}) < \infty\}$. Assume that $-\lambda < -4Q_1$, or equivalently $\lambda/4 > Q_1$. Then, by Theorem \ref{est_for_N_supnorm}, we have
		\begin{align*}
			N(-\lambda, \Hb_{\mu}) \leq \sqrt{\| q_2^\ast\|_\infty} \mathcal L(\{ x: ~q_2^\ast(x) \geq \lambda/4\}).
		\end{align*}
		By the assumption on $\lambda$, the  Lebesgue measure above is finite. Therefore, we have $N(-\lambda, \Hb_{\mu}) < \infty$, i.e., $\Lambda \geq -\lambda$. 
		
		If $-\lambda > -Q_2/4$, then $4\lambda < Q_2$. Thus,
		\begin{align*}
			N(-\lambda, \Hb_{\mu}) \geq \sqrt{\lambda} \mathcal L(\{ x: ~q_\alpha^\ast(x) \geq 4\lambda\}) = \infty,
		\end{align*}
		since the set on the right-hand side contains at least one infinite interval $(-\infty, a)$ or $(b, \infty)$. Thus, $\Lambda \leq -\lambda$. 
	\end{proof}

	\subsection{Lieb-Thirring type estimates}
	 Now we are ready to establish the \textbf{LT} estimates for any $\gamma > 0$ for an arbitrary Radon measure $\mu$. Moreover, we prove lower estimates for $\mathbf{LT}_{\gamma}$ for all $\gamma > 0$; in the special case $\gamma = \frac{1}{2}$, the lower bound can be expressed explicitly in terms of $\mu$ itself. For other values of $\gamma$, our estimates are expressed via the Otelbaev's function. Throughout this section, we assume that the negative part of the spectrum of $\Hb_{\mu}$ is discrete, that is, $q_1^*(x)\rightarrow 0 \quad$ as $x\rightarrow \pm \infty$ according to Lemma \ref{q*_and_int} and Remark \ref{rem_discr}. 
	
	We recall that for $\gamma>0$, 
	\begin{equation}\label{LT_in_terms_of_N}
		\mathrm{LT}_{\gamma}(\Hb_{\mu})\equiv    \sum_{\nu=1}^\infty |\lambda_{\nu}|^\gamma = \gamma \int_{-\infty}^{0} |\lambda|^{\gamma - 1} N(\lambda,\Hb_\mu) ~d\lambda,
	\end{equation}
	as soon as one of the quantities in \eqref{LT_in_terms_of_N} is finite.
	
	We now establish a reverse \textbf{LT} estimate for $\gamma = 1/2$. For completeness, the following theorem presents our result (a lower bound) alongside a known result (an upper bound) from \cite{HLT}. 
	
	\begin{theorem}\label{LT_estimate}
		Let $\mu$ be a Radon measure satisfying the Brinck-condition \eqref{Brinck}. Let $\Hb_\mu$ be the operator associated with the sesquilinear form \eqref{form}. Assume that the negative part of the spectrum of $\Hb_\mu$ is discrete. Let $\{\lambda_\nu\}$ be the negative eigenvalues of $\Hb_\mu$, then
		\begin{equation}\label{lt:0.5}
			\frac{\alpha}{32}  \mu(\mathbb{R})\leq \sum_{\nu=1}^\infty |\lambda_\nu|^{\frac{1}{2}} \leq \frac{1}{2}\mu(\mathbb R),
		\end{equation}
		where $\alpha = \left(\pi^2 + \frac{3}{4}\right)^{-1}$.
	\end{theorem}
	The theorem should be understood in the sense that all the occurring terms in \eqref{lt:0.5} are either simultaneously finite or infinite. 
	\begin{proof}
		The upper bound was proved in \cite{HLT}. Let us prove the lower bound. By \eqref{thm:estimate_eq2},
		\begin{align*}
			\sum_{\nu=1}^\infty |\lambda_\nu|^\gamma &\geq \gamma \int_{-\infty}^{0} |\lambda|^{\gamma - 1} \sqrt{|\lambda|} \mathcal{L} \left( \left\{x\in \mathbb{R}: q^*_{\alpha}(x) \geq 4|\lambda|\right\} \right) ~d\lambda\\
			&\geq \gamma \int_{-\infty}^{0} |\lambda|^{\gamma - \frac{1}{2}} \int_{\mathbb{R}} \chi_{\left\{x\in \mathbb{R}: q^*_{\alpha}(x) \geq 4|\lambda|\right\}}(y)~dy~d\lambda,
		\end{align*}
		where $\chi_A$ is the indicator function of the set $A$. Using Fubini's theorem, we obtain
		\begin{equation}\label{LT_type_est_lower_bound}
			\sum_{\nu=1}^\infty |\lambda_\nu|^\gamma \geq \gamma\int_{\mathbb{R}}\int_{-\frac{1}{4}q_\alpha^*(x)}^0 |\lambda|^{\gamma - \frac{1}{2}}~d\lambda ~dx = \frac{\gamma}{\gamma + \frac{1}{2}} \left(\frac{1}{4}\right)^{\gamma+\frac{1}{2}} \int_{\mathbb{R}} \left(q_\alpha^*(x)\right)^{\gamma+\frac{1}{2}} ~dx.
		\end{equation}
		We now set $\gamma = 1/2$. Let $I_k$ again be the intervals from Section \ref{sec:decomposition} and $x_k$ be their middle points. Then 
		\begin{equation*}
			\sum_{\nu=1}^\infty |\lambda_\nu|^{\frac{1}{2}} \geq \frac{1}{8} \sum_{k \in \mathbb K} \int_{I_k} q_\alpha^\ast(x)~dx
			\geq \frac{1}{8}\sum_{k \in \mathbb K: ~\mu_k(I_k) > 0} \int_{I_k} q_\alpha^\ast(x)~dx.
		\end{equation*}
		Lemma \ref{equiv_q*} implies now 
		\begin{equation*}
			\sum_{\nu=1}^\infty |\lambda_\nu|^{\frac{1}{2}} \geq \frac{1}{32} \sum_{k \in \mathbb K: ~\mu_k(I_k) > 0} |I_k| q_\alpha^*(x_k).
		\end{equation*}
		Using $q_\alpha^\ast(x_k) = \frac{1}{d_\alpha(x_k)^2}$, Lemma \ref{lemma:decomp} yields: 
		\begin{equation*}
			\sum_{\nu=1}^\infty |\lambda_\nu|^{\frac{1}{2}} \geq \frac{1}{32} \sum_{k \in \mathbb K: ~\mu_k(I_k) > 0} |I_k| \frac{1}{|I_k|^2} \geq \frac{\alpha}{32} \sum_{k \in \mathbb K: ~\mu_k(I_k) > 0} \mu_k(I_k) = \frac{\alpha}{32}  \mu(\mathbb{R}).
		\end{equation*}
		This finishes the proof.
	\end{proof}

	For a general $\gamma>0$,  the classical Lieb-Thirring estimate, where the potential is raised to the power of $\frac{1}{2} + \gamma$, seems not to have a reasonable analogy for a measure-potential, unlike $\gamma=\frac12$. Here, we establish a Lieb-Thirring-type estimate expressed in terms of the function $q_\alpha^\ast$. We want to emphasize that the following result covers the subcritical case $0 < \gamma < \frac 12$ as well.
	
	\begin{theorem}\label{LT_type_estimate}
		Let $\mu$ be a Radon measure satisfying  \eqref{Brinck} and $\Hb_\mu$ be the operator generated by quadratic form \eqref{form}. Suppose that the negative part of the spectrum of $\Hb_\mu$ is discrete. Let $\gamma>0$ and let $\{\lambda_\nu\}$ be the negative eigenvalues of $\Hb_\mu$, then
		\begin{equation*}
			\frac{\gamma}{\gamma + \frac{1}{2}} \left(\frac{1}{4}\right)^{\gamma+\frac{1}{2}}  \int_{\mathbb R} |q_{\alpha}^*(x)|^{\frac{1}{2} + \gamma} dx \leq \sum_{\nu=1}^\infty |\lambda_\nu|^\gamma \leq 4^{\gamma+1} \int_{\mathbb{R}} \left| q_\beta^*(x)\right|^{  \frac{1}{2} + \gamma}dx
		\end{equation*}
		if the integrals converge (they do this simultaneously). Otherwise, $\sum_{\nu=1}^\infty |\lambda_\nu|^\gamma = \infty$. (Here, $\alpha = \left(\pi^2 + \frac{3}{4}\right)^{-1}$ and $\beta = 2$.)
	\end{theorem}
	
	\begin{proof}
		The lower bound follows from \eqref{LT_type_est_lower_bound}. We now derive the upper bound. By using \eqref{LT_in_terms_of_N} and the upper bound in Theorem \ref{est_for_N_supnorm}, one sees that
		\begin{equation*}
			\sum_{\nu=1}^\infty |\lambda_\nu|^\gamma \leq 4\gamma \int_{-\infty}^{0} |\lambda|^{\gamma - 1}  \int_{\left\{x\in\mathbb{R}: q_\beta^*(x)\geq \frac{|\lambda|}{4}\right\}} \sqrt{q_\beta^*(y)}~dy ~d\lambda.
		\end{equation*}
		Using Fubini's theorem, we derive
		\begin{equation*}
			\sum_{\nu=1}^\infty |\lambda_
			\nu|^\gamma \leq 4\gamma \int_{\mathbb{R}} \sqrt{q_\beta^*(x)}  \int_{-4q_\beta^*(x)}^0 |\lambda|^{\gamma - 1}d\lambda dx = 4^{\gamma+1} \int_{\mathbb{R}} \left( q_\beta^*(x)\right)^{\gamma + \frac{1}{2}}dx.
		\end{equation*}
		This yields the required estimate.
	\end{proof}
	
	This theorem and Lemma \ref{lem_qalpha_qbeta_equiv} show:
	\begin{theorem}
		Let $\mu$ be a positive Radon measure satisfying \eqref{MazCond}. Then, for every $\gamma > 0$, the following holds true:
		\begin{equation*}
			C_1  \int_{\mathbb{R}} \left| q_{1}^*(x)\right|^{  \frac{1}{2} + \gamma}~dx \leq \sum_{\nu=1}^\infty |\lambda_\nu|^\gamma \leq C_2  \int_{\mathbb{R}} \left| q_{1}^*(x)\right|^{  \frac{1}{2} + \gamma}~dx,
		\end{equation*}
		where
		\begin{equation*}
			C_1 = \frac{\gamma}{\gamma + \frac{1}{2}} \left( \frac{1}{2\pi^2 + \frac{3}{2}} \right)^{1 + 2\gamma},
			\qquad
			C_2 = 2^{4\gamma + 3}.
		\end{equation*}
	\end{theorem}
	
	Next, we estimate $\mathbf{LT}_\gamma(\Hb_\mu)$ in terms of the intervals $I_k$ from Section \ref{sec:decomposition} with respect to the constant $\alpha = 2$
	
	\begin{theorem}
		Let $\mu$ be a Radon measure satisfying \eqref{MazCond}. Let $\gamma>0$ and $\{\lambda_\nu\}$ be the negative eigenvalues of $\mathbf{H}_\mu$. Then, the following holds true:
		\begin{equation*}
			\sum_{\nu=1}^\infty |\lambda_\nu|^\gamma \leq 2^{2\gamma} \sum_{k\in \mathbb{K}} \mu_k(I_k)^{2\gamma}.
		\end{equation*}
	\end{theorem}
	\begin{proof}
		In notations of Lemma \ref{N_leq_sum_N_k}, we have  
		\begin{equation*}
			\sum_{\nu=1}^\infty |\lambda_\nu|^\gamma \leq \gamma \int_{-\infty}^{0} |\lambda|^{\gamma - 1} \left(\sum_{k \in \mathbb K:~-\lambda \geq -\frac{1}{|I_k|^2}} 1\right) ~d\lambda.
		\end{equation*}
		We will use use the following notation for the shifted Heaviside function:
		\begin{equation*}
			\mathbb{H}_{-\frac{1}{|I_k|^2}}(-\lambda) :=
			\begin{cases}
				1 & \text{if } -\lambda\geq -\frac{1}{|I_k|^2},\\
				0 & \text{otherwise.}
			\end{cases}
		\end{equation*}
		Then, the last estimate can be rewritten as follows:
		\begin{align*}
			\sum_{\nu=1}^\infty |\lambda_\nu|^\gamma &\leq  \gamma \int_{-\infty}^{0} |\lambda|^{\gamma - 1} \sum_{k\in \mathbb{K}} \mathbb{H}_{-\frac{1}{|I_k|^2}}(-\lambda) ~d\lambda \leq \sum_{k\in \mathbb{K}} \gamma \int_{-\infty}^{0} |\lambda|^{\gamma - 1} \mathbb{H}_{-\frac{1}{|I_k|^2}}(-\lambda) ~d\lambda\\
			&= \sum_{k\in \mathbb{K}} \gamma \int_{-\frac{1}{|I_k|^2}}^{0} |\lambda|^{\gamma - 1}  d\lambda = \sum_{k\in \mathbb{K}} \frac{1} {|I_k|^{2\gamma}}.
		\end{align*}
		Therefore, due to Lemma \ref{lemma:decomp}, we have proven the result.
	\end{proof}
	
	\appendix
	
	\section{Comparison with preceding results}
	Here, we recall some known results and compare them with Theorem \ref{LT_type_estimate}. Throughout this section, we assume that \( q \in L_{\text{loc}}^1(\mathbb{R}) \) is a non-negative function satisfying 
	\begin{equation}\label{cond_disc}
		\int_{x}^{x+a} q(x)~dx \rightarrow 0
		\quad
		\text{as} \quad |x|\rightarrow \infty
	\end{equation}
	for any fixed \( a > 0 \). Let $\mu$ be the measure defined by
	\begin{equation}\label{mu_q}
		\mu(A) = \int_{A}q(x)dx
	\end{equation}
	for a Borel set $A\subset \mathbb{R}$. Due to \eqref{cond_disc}, the negative part of the spectrum of $\Hb_\mu$ consists of negative eigenvalues $\{\lambda_\nu(q)\}$.
	\subsection{Comparison with the Lieb-Thirring estimate}
	We assume that $\gamma > 1/2$ and use notation $p=1/2 + \gamma$.  Since the measure associated with $q$ is continuous (no point masses), the following result holds: 
	\begin{proposition}\label{avereging}
		Let $\alpha >0$ and $q \in L_{\text{loc}}^1(\mathbb{R})$ be a non-negative function. Then, 
		\begin{equation*}
			q_\alpha^*(x) = \frac{\alpha}{d_\alpha(x)} \int_{x - d_\alpha(x)/2}^{x + d_\alpha(x)/2} q(y)dy.
		\end{equation*}
	\end{proposition}
	\begin{proof}
		For a  fixed $x\in \mathbb{R}$, consider two monotonous function of the variable $d\in(0,\infty)$:
		\begin{equation*}
			f_x(d): = \int_{x-d/2}^{x+d/2} q(y)dy,
			\qquad
			h_x(d): = \frac{1}{\alpha d}. 
		\end{equation*}
		These functions are continuous in $d$, moreover, the second one is strictly decaying. For $d$ small enough,  $f_x(d) < h_x(d)$, while for $d$ large enough, $f_x(d) > h_x(d)$. Thus, there exists a unique solution $d_x$ to the equation $f_x(d_x) = h_x(d_x)$. By the definition of $q_\alpha^*$, we have $q_\alpha^*(x) = 1/d_x^2$, that is $d_\alpha(x) = d_x$. Then,
		\begin{equation*}
			q_\alpha^*(x) = \frac{1}{d_\alpha(x)^2} =  \frac{\alpha}{d_\alpha(x)} \int_{x - d_\alpha(x)/2}^{x + d_\alpha(x)/2} q(y)~dy.
		\end{equation*}
		This completes the proof.
	\end{proof}
	The following result immediately follows from the classical Lieb-Thirring estimate, Lemma \ref{equiv_q*}, and the lower bound in Theorem \ref{LT_type_estimate}. However, we would like to give an alternative proof.
	\begin{theorem}
		Let $\alpha>0$, $1<p<\infty$, then there exists  $C_{p,\alpha}>0$ depending only on $p$ and $\alpha$ such that for any non-negative function $q\in L^{p}(\mathbb{R})$,  the estimate holds
		\begin{equation*}
			\|q_{\alpha,\mu}^*\|_{L^{p}(\mathbb{R})} \leq C_{p,\alpha} \|q\|_{L^{p}(\mathbb{R})}.
		\end{equation*}
	\end{theorem}
	\begin{proof}
		We recall the definition of the Hardy-Littlewood maximal function:
		\begin{equation*}
			Mf(x) = \sup_{r>0} \frac{1}{|B(x,r)|} \int_{B(x,r)} |f(x)| dx.
		\end{equation*}
		Using Proposition \ref{avereging}, we obtain the estimate $q_\alpha^*(x) \leq \alpha Mq(x)$. Therefore, using the Hardy–Littlewood maximal inequality, we derive
		\begin{equation*}
			\|q_{\alpha,\mu}^*\|_{L^{p}(\mathbb{R})} \leq \alpha \|Mq\|_{L^p(\mathbb{R})} \leq \alpha C_{p} \|q\|_{L^{p}(\mathbb{R})}.
		\end{equation*}
		This completes the proof.
	\end{proof}
	
	The converse estimate does not hold. Consider the following example.
	
	\begin{example}\label{contr_examp_LT}
		Let $p>1$ and $x_k$ be the center point of $ \mathcal{F}_{k}^+ = [2^{k-1}, 2^k]$. We define
		\begin{equation*}
			q(x) = 
			\begin{cases}
				2^{\frac{k}{p}}, \quad &x \in \left(x_k - \frac{1}{2^{k + 1}},  x_k + \frac{1}{2^{k + 1}} \right) \text{ and } k \text{ is even},\\
				0, \quad & \text{otherwise}.
			\end{cases}
		\end{equation*}
	\end{example}
	
	\begin{proposition}\label{comp_LT}
		Let $\alpha>0$ and $1<p<\infty$. Then, for  $q\in L_{loc}^1(\mathbb{R})$, just defined, 
		\begin{equation*}
			\|q\|_{L^p(\mathbb{R})} = \infty
			\quad
			\text{and}
			\quad
			\|q_{\alpha,\mu}^*\|_{L^p(\mathbb{R})} < \infty,
		\end{equation*}
		where $\mu$ is the measure given by \eqref{mu_q}.
	\end{proposition}
	
	\begin{proof}
		Due to Lemma \ref{equiv_q*}, it is enough to prove this for $\alpha=1$. For our function $q$, we clearly have $\|q\|_{L^p(\mathcal{F}_k^+)} = 1$ for an even $k$. Hence, we obtain the first statement.
		
		To show the second statement, we write
		\begin{equation}\label{Lp_q*_LT}
			\|q_{1,\mu}^*\|_{L^p(\mathbb{R})}^p  = 	\|q_{1,\mu}^*\|_{L^p((-\infty,1))}^p  +  \sum_{k=2l-1; \; l\in\mathbb{N}}	\|q_{1,\mu}^*\|_{L^p(\mathcal{F}_k^+)}^p +  \sum_{k=2l; \; l\in\mathbb{N}}	\|q_{1,\mu}^*\|_{L^p(\mathcal{F}_k^+)}^p.
		\end{equation}
		Since
		\begin{equation}\label{est_q*_supp}
			q_{1,\mu}^*(x) \leq \frac{1}{\mathrm{dist}(x, \mathrm{supp}(q))^2},
		\end{equation}
		we have
		\begin{equation*}
			q_{1,\mu}^*(x) \leq  \frac{1}{(2-x)^{2}}
			\quad
			\text{for} \quad x\in (-\infty,1),
		\end{equation*}
		and hence, $	\|q_{1,\mu}^*\|_{L^p((-\infty,1))}<\infty$. 
		
		Next, we assume that $k$ is odd. Then
		$\mathrm{dist}(\mathcal{F}_{k}^+, \mathrm{supp}(q)) \asymp2^k.$
		Therefore, by \eqref{est_q*_supp},
		$$	\|q_{1,\mu}^*\|_{L^p(\mathcal{F}_k^+)}^p \asymp \frac{1}{2^{k(2p-1)}},$$
		and since $p>1$, the second term on the right-hand side of \eqref{Lp_q*_LT} is finite.

		It remains to show that the third term on the right-hand side of \eqref{Lp_q*_LT} is finite, thus, $k$ is even. We will estimate $q_{1,\mu}^*$ on the following two subsets of $\mathcal{F}_k^+$:
		\begin{equation}\label{sets_contr_LT}
			I_k = \left(2^{k-1}, x_k - \frac{d_{1,\mu}(x_k)}{2} + \frac{1}{2^{k+1}} \right),
			\qquad
			J_k = \left(x_k - \frac{d_{1,\mu}(x_k)}{2} + \frac{1}{2^{k+1}}, x_k \right).
		\end{equation}
		This  gives  the estimate of $L^p$-norm of $q_{1,\mu}^*$ over $[2^{k-1},x_k]\subset \mathcal{F}_k^+$. The remaining part of $\mathcal{F}_k^+$, the set $[x_k, 2^k]$, can be treated similarly.
		We first integrate $(q_{1,\mu}^*)^p$ over $J_k$.  Since $p>1$, we know that $2^{k(1 - 1/p) - 1} > 1/2^{k+1}$ for sufficiently large $k$. Therefore, 
		\begin{equation*}
			\mu \left( \Delta_{x_k} \left(2^{k     \left( 1- \frac{1}{p}\right) }   \right) \right) 
			= \int_{x_k- 2^{k(1 - 1/p) - 1}}^{x_k + 2^{k(1 - 1/p) - 1}} q(y) ~dy
			=  \frac{1}{2^{k     \left( 1- \frac{1}{p}\right) }}.
		\end{equation*}
		We conclude that $d_{1,\mu}(x_k) = 2^{k     \left( 1- \frac{1}{p}\right) }$, see Remark \ref{rem_cont_meas}. Let $x\in J_k$.	Then,
		\begin{equation*}
			x  - \frac{d_{1,\mu}(x_k) }{2} < x_k  - \frac{1}{2^{k+1}},
			\qquad
			x_k + \frac{1}{2^{k+1}}  <  x +\frac{d_{1,\mu}(x_k) }{2},
		\end{equation*}
		for sufficiently large $k$. Therefore,
		\begin{equation*}
			\mu \left( \Delta_{x} \left(d_{1,\mu}(x_k)   \right) \right) 
			= \int_{x  - \frac{d_{1,\mu}(x_k) }{2} }^{x  +  \frac{d_{1,\mu}(x_k) }{2} } q(y) ~dy
			= \int_{x_k- \frac{1}{2^{k+1}}}^{x_k + \frac{1}{2^{k+1}}} 2^{\frac{k}{p}} ~dy  = \frac{1}{d_{1,\mu}(x_k)}.
		\end{equation*}
		By Remark \ref{rem_cont_meas}, $d_{1,\mu}\arrowvert_{J_k} = 2^{k(1 - 1/p)}$ and 
		\begin{equation}\label{Lp_Jk_LT}
			\|q_{1,\mu}^*\|_{L^p(J_k)}^p  = \frac{1}{2^{k\left(2p - 3 +\frac{1}{p}\right)}}.
		\end{equation}
		
		Next, we evaluate $q_{1,\mu}^*$ for $x\in I_k$: since
		\begin{equation*}
			\mathrm{dist}(x,\mathrm{supp}(q)) = \mathrm{dist}\left(x, \left[x^k -\frac{1}{2^{k+1}},x^k +\frac{1}{2^{k+1}}\right]\right).
		\end{equation*}
		\eqref{est_q*_supp} implies:
		\begin{equation*}
			q_{1,\mu}^*(x) \leq  \frac{1}{\left(x_k - \frac{1}{2^{k+1}} - x\right)^2},
			\qquad
			\text{for }
			x\in I_k,
		\end{equation*}
		and hence,
		\begin{equation*}
			\|q_{1,\mu}^*\|_{L^p(I_k)}^p \lesssim   \frac{1}{2^{k\left(2p - 3 +\frac{1}{p}\right)}}.
		\end{equation*}
		Combining this with \eqref{Lp_Jk_LT}, we obtain the above estimate for the interval $[2^{k-1},x_k]$. By the similar arguments, we study $q_{1,\mu}^*$ on $[x_k,2^k]$ and obtain 
		\begin{equation*}
			\|q_{1,\mu}^*\|_{L^p(\mathcal{F}_k^+)}^p \lesssim \frac{1}{2^{k\left(2p - 3 +\frac{1}{p}\right)}}.
		\end{equation*}
		Since $p>1$,  $2p - 3 +\frac{1}{p} > 0$,
		therefore, the third term on the right-hand side of \eqref{Lp_q*_LT} is finite, and hence, $\|q_{1,\mu}^*\|_{L^p(\mathbb{R})}$ is finite.
	\end{proof}
	\begin{remark}
		The above results show that our version of the Lieb-Thirring estimate yields finite bounds for absolutely continuous potentials as long as the classical Lieb-Thirring estimate does. But for the potential $q$ in the above example, our bound gives a finite bound, while the classical Lieb-Thirring estimate yields no information.
	\end{remark}
	
	\subsection{Comparison with the subcritical  Netrusov-Weidl estimate}
	We now consider the subcritical case, that is, $0<\gamma<1/2$.  For $q\in L_{\text{loc}}^1(\mathbb{R})$,  we define
	\begin{equation*}
		A_\gamma(q) =  \sum_{k=0}^{\infty} \left(\int_{\mathcal{F}_k} (1+|x|)^\sigma q(x)dx\right)^{\frac{1}{2}+\gamma} + \left(\sum_{k=0}^{\infty} \left(\int_{\mathcal{F}_k} (1+|x|)^\sigma q(x) dx\right)^{\frac{1}{2}+\gamma}\right)^{\frac{2\gamma}{\frac{1}{2} + \gamma}},
	\end{equation*}
	\begin{equation*}
		B_\gamma(q)= \sum_{k=0}^{\infty} \left(\int_{k}^{k+1} q(x)dx\right) ^{2\gamma} + \sum_{k=0}^\infty \left(\int_{k}^{k+1} q(x)dx\right) ^{\frac{1}{2}+\gamma},
	\end{equation*}
	where $\sigma = \frac{1/2 - \gamma}{1/2 + \gamma}$. We  recall the following essential result; see \cite{NetrusovWeidl}:
	\begin{theorem}[Netrusov, Weidl]\label{thm_NW}
		Let $0<\gamma<1/2$ and $\sigma = \frac{1/2 - \gamma}{1/2 + \gamma}$. Define the sets $\mathcal{F}_0 = [-1,1]$ and $\mathcal{F}_k = [-2^k,-2^{k-1}]\cup[2^{k-1},2^k]$, for $k\in \mathbb{N}.$ Then, there exists $C>0$ depending only on $\gamma$ such that 
		\begin{equation*}
			\sum_{\nu\in \mathbb{N}} |\lambda_\nu(q)|^\gamma \leq C \min\left\{A_\gamma(q), B_\gamma(q)\right\},
			\quad
			\text{for any } q\in L_{\text{loc}}^1(\mathbb{R}).
		\end{equation*}
	\end{theorem}
	
	\begin{remark}
		Here, we stated special versions of Theorems 3 and 4 from \cite{NetrusovWeidl} for the one-dimensional, non-fractional case.
	\end{remark}
	
	Next, we note that the combination of Theorem \ref{thm_NW}, the lower bound in Theorem \ref{LT_type_estimate}, and Lemma \ref{lem_qalpha_qbeta_equiv} gives the following estimate:

	\begin{corollary}
		Let $\alpha>0$, and $\gamma$, $\sigma$, $\mathcal{F}_k$, $A_\gamma(q)$, $B_\gamma(q)$ be as in Theorem \ref{thm_NW}.  Then, there exists $C>0$ depending only on $\alpha$ and $\gamma$ such that 
		\begin{equation*}
			\int_{\mathbb R} |q_{\alpha,\mu}^*(x)|^{\frac{1}{2} + \gamma} dx \leq C \min\left\{A_\gamma(q), B_\gamma(q)\right\},
		\end{equation*}
		for any $q\in L_{\text{loc}}^1(\mathbb{R})$ satisfying \eqref{cond_disc}, where $\mu$ is the measure given by \eqref{mu_q}. 
	\end{corollary}
	The converse relation does not hold; see the construction below.
	
	\begin{example}\label{exm_for_NW_thm}
		Let $0<\gamma<1/2$, $\sigma= \frac{1/2 - \gamma}{1/2 + \gamma}$, and $x_k$ be the center point of $ \mathcal{F}_{k}^+ = [2^{k-1}, 2^k]$. Define
		a function $	q = q_1 + q_2$,  where 
		\begin{align*}
			q_1(x) &= 
			\begin{cases}
				1, \quad &x \in \left(x_k - \frac{1}{2^{k\sigma + 1}},  x_k + \frac{1}{2^{k\sigma + 1}} \right) \text{ and } k \text{ is even},\\
				0 & \text{otherwise},
			\end{cases}\\
			q_2(x) &= 
			\begin{cases}
				\frac{2^3}{2^{\frac{k}{2\gamma}}} & x \in \left(x_k - 2^{k-3},  x_k + 2^{k-3} \right)\text{ and } k \text{ is even},\\
				0 & \text{otherwise}.
			\end{cases}
		\end{align*}	
	\end{example}
	
	We show that for the example above, the upper bounds in Theorem \ref{thm_NW} are infinite, while the upper bound in Theorem \ref{LT_type_estimate} is finite. 
	\begin{proposition}\label{prop_NW_inf}
		Let $\alpha>0$, then there exists $q\in L_{loc}^1(\mathbb{R})$ such that
		\begin{equation*}
			A_\gamma(q) = \infty, 
			\qquad
			B_\gamma(q) = \infty,
			\qquad
			\|q_{1,\mu}^*\|_{L^{\frac{1}{2}+\gamma}(\mathbb{R})}<\infty,
		\end{equation*}
		where $\mu$ is the measure given by \eqref{mu_q}.
	\end{proposition}
	
	\begin{proof}
		Due to Lemma \ref{equiv_q*}, it suffices to prove the statement for $\alpha=1$ only. Let $q$ be the function defined in Example \ref{exm_for_NW_thm}. For an even $k$, it is clear that
		\begin{equation*}
			\int_{2^{k-1}}^{2^k} (1 + x)^\sigma q_1(x) ~dx \gtrsim 1,
			\qquad
			\sum_{m=2^{k-1}}^{2^{k}-1} \left(\int_m^{m+1} q_2(x) ~dx\right)^{2\gamma} \gtrsim 1.
		\end{equation*}
		Therefore, the first two statements are true. 
		
		Let $\mu_1$ and $\mu_2$ be the measures defined by \eqref{mu_q} with $q$ replaced by $q_1$ and $q_2$, respectively. Then, by Lemma \ref{lem_est_lin}, we know that 
		\begin{equation*}
			\|q_{1,\mu}^*\|_{L^{\frac{1}{2}+\gamma}(\mathbb{R})} \lesssim \|q_{1,\mu_1}^*\|_{L^{\frac{1}{2}+\gamma}(\mathbb{R})} + \|q_{1,\mu_2}^*\|_{L^{\frac{1}{2}+\gamma}(\mathbb{R})}.
		\end{equation*}
		We write, for $m=1,2$,
		\begin{multline}\label{Lp_q*_NW}
			\|q_{1,\mu_m}^*\|_{L^{\frac{1}{2}+\gamma}(\mathbb{R})}^{\frac{1}{2}+\gamma} \\
			= 	\|q_{1,\mu_m}^*\|_{L^{\frac{1}{2}+\gamma}((-\infty,1))}^{\frac{1}{2}+\gamma}  +  \sum_{k=2l-1; \; l\in\mathbb{N}}	\|q_{1,\mu_m}^*\|_{L^{\frac{1}{2}+\gamma}(\mathcal{F}_k^+)}^{\frac{1}{2}+\gamma} +  \sum_{k=2l; \; l\in\mathbb{N}}	\|q_{1,\mu_m}^*\|_{L^{\frac{1}{2}+\gamma}(\mathcal{F}_k^+)}^{\frac{1}{2}+\gamma}.
		\end{multline}
		By the same arguments we do in the proof of Proposition \ref{comp_LT}, the first term on the right-hand side is finite for $m=1,2.$
		
		To investigate the second term, we consider odd $k$ first. For $m=1,2$, we estimate
		\begin{equation*}
			\mathrm{dist}(\mathcal{F}_k^+, \mathrm{supp}(q_m)) \leq \mathrm{dist}(\mathcal{F}_k^+, \mathrm{supp}(q))  \asymp 2^k.
		\end{equation*}
		Hence, \eqref{est_q*_supp} implies 
		\begin{equation*}
			\|q_{1,\mu_m}^*\|_{L^{\frac{1}{2}+\gamma}(\mathcal{F}_k^+)}^{\frac{1}{2}+\gamma} \asymp 2^{k-1} \frac{1}{(2^{2k})^{\frac{1}{2} + \gamma}} \asymp \frac{1}{2^{2k\gamma}}.
		\end{equation*}
		Therefore, the second term in \eqref{Lp_q*_NW} on the right-hand side of \eqref{Lp_q*_NW}is finite. 
		
		It remains to show that the third term is finite for $m=1,2$, this means, for an even $k$. We  consider  the potential  $q_{1,\mu_1}$ on the following subsets in $\mathcal{F}_k^+$:
		\begin{equation*}
			I_k^1 = \left(2^{k-1}, x_k - \frac{d_{1,\mu_1}(x_k)}{2} + \frac{1}{2^{k\sigma+1}}\right),
			\qquad
			J_k^1 = \left(x_k - \frac{d_{1,\mu_1}(x_k)}{2} + \frac{1}{2^{k\sigma+1}}, x_k \right).
		\end{equation*}
		We study $q_{1,\mu}^*$ in $J_k^1$ first. Since $0<\sigma<1$, for sufficiently large $k$,
		\begin{equation*}
			\left(x_k - \frac{1}{2^{k\sigma + 1}},  x_k + \frac{1}{2^{k\sigma + 1}} \right) \subset \left(x_k - 2^{k\sigma - 1}, x_k + 2^{k\sigma - 1}\right) \subset \mathcal{F}_k^+.
		\end{equation*}
		Thus, $\mu_1(\Delta_{x_k}(2^{k\sigma}))  = {2^{-k\sigma}},$
		and hence, $d_{1,\mu_1}(x_k) = 2^{k\sigma}$, see Remark \ref{rem_cont_meas}. For $x\in J_k^1$,
		\begin{equation*}
			x - 2^{k\sigma - 1}< x_k -\frac{1}{2^{k\sigma+1}},
			\qquad
			x_k +\frac{1}{2^{k\sigma+1}} < x + \frac{d_{1,\mu_1}(x_k)}{2}  = x + 2^{k\sigma - 1}
		\end{equation*}
		and hence,
		$\mu_1(\Delta_{x} (2^{k\sigma})) =2^{-k\sigma}$.
		Therefore, as in  Remark \ref{rem_cont_meas}, $d_{1,\mu_1}\arrowvert_{J_k^1} = 2^{k\sigma}$, so that
		\begin{equation*}
			\|q_{1,\mu_1}^*\|_{L^{\frac{1}{2}+\gamma}(J_k^1)}^{{\frac{1}{2}+\gamma}} = \left( \frac{d_{1,\mu_1}(x_k)}{2}  - \frac{1}{2^{k\sigma+1}}\right) \frac{1}{2^{2k\sigma\left(\frac{1}{2} + \gamma \right)}} \asymp  \frac{1}{2^{2k\sigma\gamma}}.
		\end{equation*}
		
		Let $x\in I_k^1$, then
		\begin{equation}\label{Lp_Jk1_NW}
			\mathrm{dist}(x,\mathrm{supp}(q_1)) = x_k - \frac{1}{2^{k\sigma + 1}} - x.
		\end{equation}
		Now, \eqref{est_q*_supp} implies 
		\begin{equation*}
			\|q_{1,\mu_1}^*\|_{L^{\frac{1}{2}+\gamma}(I_k^1)}^{\frac{1}{2}+\gamma} \leq \int_{I_k^1} \frac{1}{\left(x_k - \frac{1}{2^{k\sigma + 1}} - x\right)^{1+2\gamma}}~dx \leq \frac{1}{\left( \frac{d_{1,\mu_1}(x_k)}{2} - \frac{1}{2^{k\sigma}}   \right)^{2\gamma}} \lesssim \frac{1}{2^{2k\sigma\gamma}}.
		\end{equation*}
		Combining this with \eqref{Lp_Jk1_NW}, we obtain
		\begin{equation*}
			\|q_{1,\mu_1}^*\|_{L^{\frac{1}{2}+\gamma}(2^{k-1},x_k)}^{\frac{1}{2}+\gamma}  \lesssim \frac{1}{2^{2k\sigma\gamma}}.
		\end{equation*}
		In the same way, we investigate $q_{1,\mu}^*$ on the interval  $[x_k,2^k]$ and derive 
		\begin{equation*}
			\|q_{1,\mu_1}^*\|_{L^{\frac{1}{2}+\gamma}(\mathcal{F}_k^+)}^{\frac{1}{2}+\gamma}  \lesssim \frac{1}{2^{2k\sigma\gamma}}.
		\end{equation*}
		Since $\gamma<1/2$, the latter estimates show that the third term on the right-hand side in \eqref{Lp_q*_NW} is finite for $m=1$. We will show that this is also true for $m=2$.  To study $q_{1,\mu_2}^*$ on $\mathcal{F}_k^+$ with an even $k$, we will consider two cases: $\gamma\leq 1/4$ and $\gamma>1/4$.
		
		Let $\gamma\leq 1/4$. For $x\in \mathcal{F}_k^+$, 
		\begin{equation*}
			(x - 2^{k-2}, x + 2^{k-2}) \subset \mathcal{F}_{k-1}^+ \cup\mathcal{F}_{k}^+\cup \mathcal{F}_{k+1}^+
		\end{equation*}
		and $q_{1,\mu_2}^*\arrowvert_{\mathcal{F}_{k-1}^+\cup \mathcal{F}_{k+1}^+} =0$. Therefore, using $\gamma\leq 1/4$, we estimate
		\begin{equation*}
			\int_{x - 2^{k-2}}^{x+2^{k-2}} q_2(y)~dy \leq 2^{k-2} \frac{2^3}{2^{\frac{k}{2\gamma}}} \leq  \frac{1}{2^{k-1}}
		\end{equation*}
		for sufficiently large $k$. Then, $d_{1,\mu_2}(x)>2^{k-1}$ for $x\in \mathcal{F}_k^+$ and therefore: 
		\begin{equation}\label{est_gamma_leq_14}
			\|q_{1,\mu_2}^*\|_{L^{\frac{1}{2}  +  \gamma  }(\mathcal{F}_k^+)}^{ \frac{1}{2}  +  \gamma }\lesssim \frac{1}{2^{2k\gamma}}.
		\end{equation}
		Hence,  the third term on the right-hand side of \eqref{Lp_q*_NW} is finite for $m=2$.
		
		Finally, for $\gamma>1/4$, 
		\begin{equation*}
			\int_{x_k - 2^{k-2}}^{x_k+2^{k-2}} q_2(x)dx = 2^{k-2} \frac{2^3}{2^{\frac{k}{2\gamma}}} \geq \frac{1}{2^{k-1}}
		\end{equation*}
		for sufficiently large $k$, and hence, $d_{1,\mu_2}(x_k)< 2^{k-1}$. This means that for $x\in \mathcal{F}_k^+$,
		\begin{equation*}
			\left(x - \frac{d_{1,\mu_2}(x_k)}{2}, x + \frac{d_{1,\mu_2}(x_k)}{2}\right) \subset\mathcal{F}_{k-1}^+ \cup\mathcal{F}_{k}^+\cup \mathcal{F}_{k+1}^+,
		\end{equation*}
		for sufficiently large $k$. Therefore, since $q_{1,\mu_2}^*\arrowvert_{\mathcal{F}_{k-1}^+\cup \mathcal{F}_{k+1}^+} =0$, we estimate
		\begin{equation*}
			\frac{1}{d_{1,\mu_2}(x)} = \int_{x-\frac{d_{1,\mu_2}(x)}{2}}^{x+\frac{d_{1,\mu_2}(x)}{2}} q_2(y)dy \lesssim d_{1,\mu_2}(x) \frac{1}{2^{\frac{k}{2\gamma}}},
			\quad
			\text{for } x\in \mathcal{F}_k^+,
		\end{equation*}
		so that $q_{1,\mu_2}^*(x) =  1/d_{1,\mu_2}^2(x)\lesssim1/2^{\frac{k}{2\gamma}}$. Therefore,
		\begin{equation*}
			\|q_{1,\mu_2}^*\|_{L^{\frac{1}{2}  +  \gamma  }(\mathcal{F}_k^+)}^{ \frac{1}{2}  +  \gamma }  \lesssim 2^k \frac{1}{2^{\frac{k}{2\gamma}(\frac{1}{2}+\gamma)}}.
		\end{equation*}
		Since $\frac{1}{2\gamma} \left(\frac{1}{2} + \gamma\right)\ge 1$ for $\gamma<1/2$,
		we conclude from this and \eqref{est_gamma_leq_14} that for both $\gamma\leq 1/4$ and $\gamma>1/4$,  the third term on the right-hand side in \eqref{Lp_q*_NW} is finite for $m=2$. This completes the proof.
	\end{proof}
	
	\begin{remark}
		For $q$ as in Example \ref{exm_for_NW_thm}, Theorem \ref{thm_NW}  gives no information due to Proposition \ref{prop_NW_inf}. However, from Theorem \ref{LT_type_estimate} and Proposition \ref{prop_NW_inf}, we conclude that $\mathrm{LT}_{\gamma}(\mathbf{H}_\mu)$ is finite, which shows that our estimates are sharper.
	\end{remark}

	\section{Examples}
	In this section, we obtain some corollaries and give some examples.
	
	Let $(x_k)_{k \in \mathbb Z}$ be a sequence tending to $\pm \infty$ as $k$ tends to $\pm \infty$. Further, let $(a_k)_{k \in \mathbb Z}$ be a sequence of positive numbers. Consider the measure $\mu=\sum_{k\in \mathbb{Z}} a_k\delta_{ x_k}$. Of course, the quadratic form $\mathbf{h}_\mu$ is not lower-bounded when the sequence $(a_k)_{k \in \mathbb Z}$ is unbounded. Hence, we consider examples with case $a_k \to 0$ in the following.

	\begin{corollary} \label{l3}
		Let $\mu=\sum_{k=1}^\infty a_k\delta_{x_k}$. Assume that there exists $\alpha>0$ such that
		\begin{equation*}
			x_{k+1} - x_k>\max\left\{   (\alpha a_k)^{-1}, (\alpha a_{k+1})^{-1}   \right\}
		\end{equation*}
		for all $k\in\mathbb{N}$. Then,
		\begin{equation}\label{3}
			C_1 \alpha^{2\gamma}\sum_{k=1}^\infty a_k^{2\gamma}\leq\sum_{k=1}^\infty\lambda_k^\gamma\leq C_2 \alpha^{2\gamma} \sum_{k=1}^\infty a_k^{2\gamma}
		\end{equation}
		for some $C_1$, $C_2>0$ independent of $\{x_k\}$ and $\{a_k\}$.
		
	\end{corollary}
	
	\begin{proof} We note that
		$$
		q_\alpha^*(x)=
		\left\{
		\begin{array}{lll}
			(a_k\alpha)^2,& \text{if}& x\in(x_k-(2a_k\alpha)^{-1},x_k+(2a_k\alpha)^{-1})\\
			\left(\min\left\{\frac1{x-x_k} ,\frac1{x_{k+1}-x}  \right\}\right)^2,& \text{if}&  x\in(x_k+(2a_k\alpha)^{-1},x_{k+1}-(2a_{k+1}\alpha)^{-1}).
		\end{array}
		\right.
		$$
		Therefore, for $k\in\mathbb{N}$,
		\begin{equation*}
			\int_{x_k-(2a_k\alpha)^{-1}}^{x_k+(2a_k\alpha)^{-1}}(q_\alpha^*(x))^{\frac12+\gamma}~dx = (a_k\alpha)^{2\gamma}
		\end{equation*}
		and
		\begin{multline*}
			\int_{x_k+(2a_k\alpha)^{-1}}^{x_{k+1}-(2a_{k+1}\alpha)^{-1}}(q_\alpha^*(x))^{\frac12+\gamma}dx \leq \int_{-\infty}^{x_{k+1}-(2a_{k+1}\alpha)^{-1}}\left(\frac1{x_{k+1}-x}  \right)^{1+2\gamma}dx\\
			+  \int_{x_k+(2a_k\alpha)^{-1}}^\infty\left(\frac1{x-x_{k}}  \right)^{1+2\gamma}dx
			\asymp (a_{k+1}\alpha)^{2\gamma}+(a_k\alpha)^{2\gamma}.
		\end{multline*}
		Consequently, we obtain \eqref{3}.
	\end{proof}
	We note again that for the classical Lieb-Thirring estimate \eqref{LT.initial}, when the potential can be written as $V = V_1 + V_2$, with disjoint supports, the inequality does not reflect how far these two supports are apart from each other. This behavior is fundamentally different from our Lieb-Thirring-type estimates. We demonstrate this with the following explicit example.
	\begin{example}
		Let $\mu = \delta_0 + \delta_y$, where $y > 0$. The Otelbaev function has different forms, depending on $y < \frac 14$, $\frac 14 \leq y < \frac 12$ or $y \geq \frac 12$. Explicit computations show the following: for $y \geq \frac{1}{2}$,
		\begin{align*}
			q_2^\ast(x) = \begin{cases} 
				\frac{1}{4x^2}, \quad &x \leq -\frac{1}{4},\\
				4, \quad &- \frac 14 \leq x \leq \frac 14,\\
				\frac{1}{4x^2}, \quad &\frac 14 \leq x \leq \frac{y}{2},\\
				\frac{1}{4(y-x)^2}, \quad &\frac{y}{2} \leq x \leq y -\frac{1}{4},\\
				4, \quad & y-\frac{1}{4} \leq x \leq y + \frac{1}{4},\\
				\frac{1}{4(y-x)^2}, \quad &x \geq y + \frac{1}{4}.
			\end{cases}
		\end{align*}
		Next, for  $\frac{1}{4} \leq y < \frac 12$, Otelbaev's function equals
		\begin{align*}
			q_2^\ast(x) = \begin{cases} 
				\frac{1}{4x^2}, \quad &x \leq -\frac{1}{4},\\
				4, \quad &  -\frac 14 \leq x \leq y - \frac 14,\\
				\frac{1}{4(y-x)^2}, \quad & y - \frac 14 \leq x \leq \frac{y}{2},\\
				\frac{1}{4x^2}, \quad &\frac y2 \leq x \leq \frac 14,\\
				4, \quad &  \frac14\leq x \leq y + \frac{1}{4},\\
				\frac{1}{4(y-x)^2}, \quad &x \geq y + \frac{1}{4}.
			\end{cases}
		\end{align*}
		Finally, for $0 < y < \frac 14$, 
		\begin{align*}
			q_2^\ast(x) = \begin{cases}
				\frac{1}{4x^2}, \quad &x \leq -\frac{1}{4},\\
				4, \quad &-\frac 14 \leq x \leq y - \frac{1}{4}\\
				\frac{1}{4(y-x)^2}, \quad & y-\frac 14 \leq x \leq y - \frac 18\\
				16, \quad & y - \frac 18 \leq x \leq \frac 18, \\
				\frac{1}{4x^2}, \quad & \frac 18 \leq x \leq \frac 14, \\
				4, \quad &  \frac 14 \leq x \leq y + \frac{1}{4},\\
				\frac{1}{4(y-x)^2}, \quad &x \geq y + \frac{1}{4}.
			\end{cases}
		\end{align*}
		Now, for $\gamma > 0$ we have the following expressions for the Lieb-Thirring estimates from Theorem \ref{LT_type_estimate}: For $y \geq \frac 12$, it is
		\begin{align*}
			4^{\gamma +1}\int_\mathbb{R} |q_2^\ast(x)|^{\gamma + \frac 12}~dx &= \frac{4\cdot 16^\gamma }{\gamma} + 8\cdot 16^\gamma - \frac{2\cdot4^\gamma \cdot y^{-2\gamma}}{\gamma}.
		\end{align*}
		For $\frac 14 \leq y < \frac 12$, we have
		\begin{align*}
			4^{\gamma +1}\int_\mathbb{R} |q_2^\ast(x)|^{\gamma + \frac 12}~dx &= \frac{\gamma 4^{2\gamma + 2}y+y^{-2\gamma}4^{\gamma + \frac 12}}{\gamma};
		\end{align*}
		and for $0 < y \leq \frac 14$, one sees that 
		\begin{align*}
			4^{\gamma + 1}\int_{\mathbb R} |q_2^\ast(x)|^{\gamma + \frac 12}~dx &= \frac{4^{3\gamma + \frac 12} + y\gamma 4^{2\gamma + 2} + \gamma 4^{3\gamma + \frac 52}(\frac 18 - \frac y2)}{\gamma}. 
		\end{align*}
		This calculation shows the following interesting limit behavior for the upper Lieb-Thirring estimate, which is not captured by the classical Lieb-Thirring estimates: When the two components of $\mu$ move away from each other, we obtain the behaviour:
		\begin{align*}
			4^{\gamma + 1}\int_{\mathbb R} |q_2^\ast(x)|^{\gamma + \frac{1}{2}}~dx \overset{y \to \infty} \longrightarrow \frac{4\cdot 16^\gamma }{\gamma} + 8\cdot 16^\gamma.
		\end{align*}
		On the other hand, when the compontents become close to each other, we see that:
		\begin{align*}
			4^{\gamma + 1}\int_{\mathbb R} |q_2^\ast(x)|^{\gamma + \frac 12}~dx \overset{y \to 0}{\longrightarrow} \frac{4^{3\gamma + \frac 12} + \gamma 4^{3\gamma + 1}}{\gamma}.
		\end{align*}
		Note that the limit $y \to 0$ agrees with the Lieb-Thirring bound obtained for the potential $\mu = 2\delta_0$: Indeed, for this potential,
		\begin{align*}
			q_2^\ast(x) = \begin{cases}
				16, \quad &|x|\leq \frac{1}{8},\\
				\frac{1}{4|x|^2}, \quad &|x| > \frac 18.
			\end{cases}
		\end{align*}
		It is not hard to verify that, for this function, the integral $4^{\gamma + 1}\int_{\mathbb R}|q_2^\ast(x)|^{\gamma + \frac 12}~dx$ agrees with the above value. Of course, this is no coincidence, as it can easily be verified by the continuity result \cite[Prop.\ III.5]{FulscheNursultanov} and the dominated convergence theorem. 
		
		On the other hand, for the limit $y \to \infty$, one might expect to recover the estimates for the potential $\delta_0$. But this is indeed not the case. Note that Otelbaev's function for the potential $\delta_0$ is given by:
		\begin{align*}
			q_2^\ast(x) = \begin{cases}
				4, \quad &|x| \leq \frac{1}{4},\\
				\frac{1}{4|x|^2}, \quad &|x| > \frac 14.
			\end{cases}
		\end{align*}
		With this explicit form, it is simple to verify that
		\begin{align*}
			4^{\gamma + 1} \int_{\mathbb R}|q_2^\ast(x)|^{\gamma + \frac{1}{2}}~dx &= \frac{4\gamma+2}{\gamma}4^{2\gamma}.
		\end{align*}
		This is clearly not the limit that we obtained for $y \to \infty$ above.
	\end{example}

	\begin{example}\label{ex:compsupp}
		Assume that the measure $\mu$ is compactly supported, say, $[-R,R] = \operatorname{conv}(\operatorname{supp}(\mu))$, the convex hull of the support of $\mu$.   Then, for any $|x| \geq R + \frac{1}{2\alpha \mu(R)}$,
		\begin{align*}
			q_\alpha^*(x) \leq \frac{1}{\operatorname{dist}^2(x, \operatorname{supp}(\mu))} \leq \frac{1}{(x - R)^2}.
		\end{align*}
		Also, for any $x\in\mathbb{R}$, by using Lemma \ref{max_interval}, we estimate
		\begin{equation*}
			\mu(\mathbb{R}) \geq \mu(\Delta_x(d_\alpha(x))) \geq \frac{1}{\alpha d_\alpha(x)} = \frac{1}{\alpha} \sqrt{q_\alpha^*(x)}.
		\end{equation*}
		Therefore, 
		\begin{align*}
			\int_{\mathbb{R}}|q_\alpha^*(x)|^{\frac{1}{2}+\gamma}dx &\leq 2\int_0^{R + \frac{1}{2\alpha \mu(R)}} (\alpha \mu(\mathbb{R}))^{1+2\gamma}dx + 2\int_{R + \frac{1}{2\alpha \mu(R)}}^\infty \frac{1}{(x - R)^{1+2\gamma}}dx\\
			& \leq R\alpha^{1+2\gamma} \mu(\mathbb{R})^{1+2\gamma} + \alpha^{2\gamma} \left(\frac{1}{2} + \frac{2^{2\gamma - 1}}{\gamma}\right)\mu(\mathbb{R})^{2\gamma}.
		\end{align*}
		Hence, due to Theorem \ref{LT_type_estimate}, 
		\begin{equation*}
			\sum_{k=1}^\infty |\lambda_k|^\gamma \leq 2^{3+4\gamma} R \mu(\mathbb{R})^{1+2\gamma} + 2^{4\gamma + 1} \left(1 + \frac{2^{2\gamma}}{\gamma}\right)\mu(\mathbb{R})^{2\gamma}.
		\end{equation*}
	\end{example}

\section*{Acknowledgement}
This research of the second author was funded by the Science Committee of the Ministry of Education and Science of the Republic of Kazakhstan (Grant No. AP22683207). Also, the second author was supported by the European Research Council of the European Union, grant 101086697 (LoCal), and the Research Council of Finland, grants 347715 and 353096. Views and opinions expressed are those of the authors only and do not necessarily reflect those of the European Union or the other funding organizations. Neither the European Union nor the other funding organizations can be held responsible for them.

	\bibliographystyle{abbrv}
	\bibliography{Fulsche_Nursultanov_Rozenblum.bib}

\begin{thebibliography}{10}

\bibitem{Albev.Gesztesy}
S.~Albeverio, F.~Gesztesy, R.~H\o{}egh-Krohn, and W.~Kirsch.
\newblock On point interactions in one dimension.
\newblock {\em J. Operator Theory}, 12(1):101--126, 1984.

\bibitem{AlKosh}
S.~Albeverio and V.~Koshmanenko.
\newblock On {S}chr\"{o}dinger operators perturbed by fractal potentials.
\newblock {\em Rep. Math. Phys.}, 45(3):307--326, 2000.

\bibitem{BFS23}
S.~Bachmann, R.~Froese, and S.~Schraven.
\newblock Counting eigenvalues of {S}chr\"{o}dinger operators using the
  landscape function.
\newblock {\em J. Spectr. Theory}, 13(4):1445--1472, 2023.

\bibitem{BFS24}
S.~Bachmann, R.~Froese, and S.~Schraven.
\newblock Two-sided {L}ieb-{T}hirring bounds.
\newblock arXiv:2403.19023, 2024.

\bibitem{BelloniRobinett}
M.~Belloni and R.~Robinett.
\newblock The infinite well and {D}irac delta function potentials as
  pedagogical, mathematical and physical models in quantum mechanics.
\newblock {\em Phys. Rep.}, 540(2):25--122, 2014.

\bibitem{Birman.61}
M.~Birman.
\newblock On the spectrum of singular boundary-value problems.
\newblock {\em Mat. Sb. (N.S.)}, 55(97):125--174, 1961.

\bibitem{Bonin}
C.~Bonin, J.~Lunardi, and L.~Manzoni.
\newblock The physical interpretation of point interactions in one-dimensional
  relativistic quantum mechanics.
\newblock {\em J. Phys. A}, 57(9):Paper No. 095204, 26, 2024.

\bibitem{BrascheExnerKuperin}
J.~Brasche, P.~Exner, Y.~Kuperin, and P.~\v{S}eba.
\newblock Schr\"{o}dinger operators with singular interactions.
\newblock {\em J. Math. Anal. Appl.}, 184(1):112--139, 1994.

\bibitem{BrascheNizhnik}
J.~Brasche and L.~Nizhnik.
\newblock One-dimensional {S}chr\"{o}dinger operators with general point
  interactions.
\newblock {\em Methods Funct. Anal. Topology}, 19(1):4--15, 2013.

\bibitem{Brinck}
I.~Brinck.
\newblock {Self-adjointness and Spectra of Sturm-Liouville operators}.
\newblock {\em Math. Scand.}, 7:219--240, 1959.

\bibitem{DaHuSi}
D.~Damanik, D.~Hundertmark, and B.~Simon.
\newblock Bound states and the {S}zeg{\"o} condition for {J}acobi matrices and
  {S}chr\"{o}dinger operators.
\newblock {\em J. Funct. Anal.}, 205(2):357--379, 2003.

\bibitem{DemkovOstrovskii}
Y.~Demkov and V.~Ostrovskii.
\newblock {\em Zero-Range Potentials and Their Applications in Atomic Physics}.
\newblock Penum Press. New York and London, 1985.

\bibitem{FrankLaptev2008}
R.~Frank and A.~Laptev.
\newblock Spectral inequalities for {S}chr\"{o}dinger operators with surface
  potentials.
\newblock In {\em Spectral theory of differential operators}, volume 225 of
  {\em Amer. Math. Soc. Transl. Ser. 2}, pages 91--102. Amer. Math. Soc.,
  Providence, RI, 2008.

\bibitem{FrankLaptevWeidl}
R.~Frank, A.~Laptev, and T.~Weidl.
\newblock {\em Schr\"{o}dinger operators: eigenvalues and {L}ieb-{T}hirring
  inequalities}, volume 200 of {\em Cambridge Studies in Advanced Mathematics}.
\newblock Cambridge University Press, Cambridge, 2023.

\bibitem{FulscheNursultanov}
R.~Fulsche and M.~Nursultanov.
\newblock Spectral theory for {S}turm-{L}iouville operators with measure
  potentials through {O}telbaev's function.
\newblock {\em J. Math. Phys.}, 63(1):Paper No. 012101, 22, 2022.

\bibitem{HLT}
D.~Hundertmark, E.~H. Lieb, and L.~E. Thomas.
\newblock A sharp bound for an eigenvalue moment of the one-dimensional
  {S}chr\"{o}dinger operator.
\newblock {\em Adv. Theor. Math. Phys.}, 2(4):719--731, 1998.

\bibitem{Kato}
T.~Kato.
\newblock {\em Perturbation Theory for Linear Operators}.
\newblock Springer, 2nd edition, 1976.

\bibitem{KostenkoMalamud}
A.~Kostenko and M.~Malamud.
\newblock 1-{D} {S}chr\"{o}dinger operators with local point interactions on a
  discrete set.
\newblock {\em J. Differential Equations}, 249(2):253--304, 2010.

\bibitem{MazVerb1D}
V.~Maz'ya and I.~Verbitsky.
\newblock Boundedness and compactness criteria for the one-dimensional
  {S}chr\"{o}dinger operator.
\newblock In {\em Function spaces, interpolation theory and related topics
  ({L}und, 2000)}, pages 369--382. de Gruyter, Berlin, 2002.

\bibitem{MikhailetsMolyboga}
V.~Mikhailets and V.~Molyboga.
\newblock Schr\"{o}dinger operators with measure-valued potentials:
  semiboundedness and spectrum.
\newblock {\em Methods Funct. Anal. Topology}, 24(3):240--254, 2018.

\bibitem{NetrusovWeidl}
Y.~Netrusov and T.~Weidl.
\newblock On {L}ieb-{T}hirring inequalities for higher order operators with
  critical and subcritical powers.
\newblock {\em Comm. Math. Phys.}, 182(2):355--370, 1996.

\bibitem{Nursultanov}
M.~Nursultanov.
\newblock {Spectral properties of the Schr\"{o}dinger operator with
  $\delta$-distribution}.
\newblock {\em Math Notes}, 100:263--275, 2016.

\bibitem{Otelbaev}
M.~Otelbaev.
\newblock {Estimates of the eigenvalues of singular differential operators}.
\newblock {\em Mat. Zametki}, 20:859--867, 1976.
\newblock in Russian.

\bibitem{Otelbaev2}
M.~Otelbaev.
\newblock {Estimates of the spectrum of the Sturm–Liouville operator}.
\newblock {\em Gylym, Alma-Ata}, 1990.
\newblock in Russian.

\bibitem{Reed2}
M.~Reed and B.~Simon.
\newblock {\em {Methods of Modern Mathematical Physics 2: Fourier Analysis,
  Self-Adjointness}}.
\newblock Academic Press, 1975.

\bibitem{Rozenblum2022}
G.~Rozenblum.
\newblock Lieb-{T}hirring estimates for singular measures.
\newblock {\em Ann. Henri Poincar\'{e}}, 23(11):4115--4130, 2022.

\bibitem{RozenblumTashchiyan}
G.~Rozenblum and G.~Tashchiyan.
\newblock Eigenvalues of the {B}irman-{S}chwinger operator for singular
  measures: the noncritical case.
\newblock {\em J. Funct. Anal.}, 283(12):Paper No. 109704, 42, 2022.

\bibitem{SavchukShkalikov}
A.~Savchuk and A.~Shkalikov.
\newblock Sturm-{L}iouville operators with distribution potentials.
\newblock {\em Tr. Mosk. Mat. Obs.}, 64:159--212, 2003.

\bibitem{Schmincke}
U.-W. Schmincke.
\newblock On {S}chr\"{o}dinger's factorization method for {S}turm-{L}iouville
  operators.
\newblock {\em Proc. Roy. Soc. Edinburgh Sect. A}, 80(1-2):67--84, 1978.

\bibitem{Weidl}
T.~Weidl.
\newblock On the {L}ieb-{T}hirring constants {$L_{\gamma,1}$} for {$\gamma\geq
  1/2$}.
\newblock {\em Comm. Math. Phys.}, 178(1):135--146, 1996.

\end{thebibliography}
\end{document}